\documentclass[10.5pt]{article}

\usepackage{latexsym}
\usepackage{amssymb}
\usepackage{amsthm}
\usepackage{amsmath}
\usepackage{color}
\usepackage{tikz}
\usepackage{mathrsfs}
\usepackage{titlesec}
\usepackage{tikz-cd}
\usepackage{enumitem}
\usepackage{array}
\usepackage{stmaryrd}
\usepackage{MnSymbol}
\usepackage{mathtools}
\usepackage{amsmath,calligra,mathrsfs}
\usepackage{adjustbox}
\usepackage[hang,flushmargin]{footmisc}

\DeclareMathAlphabet\mathbfcal{OMS}{cmsy}{b}{n}

\usepgflibrary{arrows}

\newtheorem{theorem}{Theorem}[section]
\newtheorem{corollary}[theorem]{Corollary}
\newtheorem{lemma}[theorem]{Lemma}
\newtheorem{proposition}[theorem]{Proposition}

\newtheorem{definition}[theorem]{Definition}
\newtheorem{remark}[theorem]{Remark}
\newtheorem{example}[theorem]{Example}

\newtheorem{notation}[theorem]{Notation}
\newtheorem{situation}[theorem]{Situation}

\titleformat*{\section}{\large\bfseries}
\titleformat*{\subsection}{\normalsize\bfseries}

\newcommand{\sHom}{\mathscr{H}\text{\kern -3pt {\calligra\large om}}\,}
\newcommand{\CC}{\mathbb{C}}
\newcommand{\QQ}{\mathbb{Q}}
\newcommand{\ZZ}{\mathbb{Z}}

\newcommand{\cQ}{\mathcal{Q}}

\newcommand{\cA}{\mathcal{A}}

\newcommand{\cG}{\mathcal{G}}

\newcommand{\cS}{\mathcal{S}}
\newcommand{\cO}{\mathcal{O}}
\newcommand{\cX}{\mathcal{X}}
\newcommand{\cT}{\mathcal{T}}
\newcommand{\cH}{\mathcal{H}}

\newcommand{\cY}{\mathcal{Y}}

\newcommand{\SL}{\mathrm{SL}}
\newcommand{\cF}{\mathcal{F}}
\newcommand{\bG}{\mathbb{G}}
\newcommand{\tbf}{\textbf}
\newcommand{\bP}{\mathbb{P}}

\newcommand{\cE}{\mathcal{E}}

\newcommand{\mrm}{\mathrm}

\newcommand{\mbf}{\mathbf}
\newcommand{\Hom}{\mrm{Hom}}

\newcommand{\Gr}{\mrm{Gr}}

\newcommand{\ttilde}{\widetilde}
\newcommand{\hhat}{\widehat}
\newcommand{\bs}{\backslash}
\newcommand{\cB}{\mathcal{B}}
\newcommand{\rk}{\mrm{rk}\,}
\newcommand{\scr}{\mathscr}
\newcommand{\mf}{\mathfrak}
\newcommand{\ul}{\underline}
\newcommand{\Sym}{\mrm{Sym}}
\newcommand{\bbar}{\overline}

\newcommand{\GL}{\mrm{GL}}
\newcommand{\mbb}{\mathbb}
\newcommand{\cL}{\mathcal{L}}
\newcommand{\cM}{\mathcal{M}}

\newcommand{\ord}{\mrm{ord}\,}

\newcommand{\Coh}{\mrm{Coh}}
\newcommand{\pr}{\mrm{pr}}
\newcommand{\Bun}{\mrm{Bun}}
\newcommand{\Spec}{\mrm{Spec}\,}
\newcommand{\Isom}{\mrm{Isom}}

\newcommand{\cD}{\mathcal{D}}
\newcommand{\cN}{\mathcal{N}}

\newcommand{\mH}{\mrm{H}}
\newcommand{\SB}{\mrm{SB}}

\newcommand{\Bh}{\mrm{Bh}}
\newcommand{\SO}{\mrm{SO}}
\newcommand{\diag}{\mrm{diag}}
\newcommand{\End}{\mrm{End}}
\newcommand{\im}{\mrm{im}}

\newcommand{\mA}{\mbb{A}}

\newcommand{\Sp}{\mrm{Sp}}

\newcommand{\Proj}{\mrm{Proj}\,}

\newcommand{\PGL}{\mrm{PGL}}

\newcommand{\mP}{\mrm{P}}

\newcommand{\Parbun}{\mrm{Parbun}}
\newcommand{\Quot}{\mrm{Quot}}

\begin{document}

\begin{center}\tbf{\textsc{COMPACTIFICATIONS OF MODULI OF\\$G$-BUNDLES AND CONFORMAL BLOCKS}}\end{center}

\begin{center}Avery Wilson\end{center}
\begin{center} Department of Mathematics \\ University of North Carolina \\ Chapel Hill, NC 27514 USA\\ avwi1407@live.unc.edu\end{center}
\medskip

\begin{center}\begin{tabular}{m{13cm}}\tbf{Abstract.} For a stable curve of genus $g\geq 2$ and simple Lie algebra of type A or C, we show that the conformal blocks algebra $\cA$ on $\bbar{\cM}_g$ is finitely generated and establish an explicit connection to Schmitt and Mu\~{n}oz-Casta\~{n}eda's compactification of the moduli space of $G$-bundles.\end{tabular}\end{center}

\section{Introduction}

Let $C_0$ be a stable curve over $\CC$ of genus $g\geq 2$, and $G$ a simple, simply connected algebraic group. When $C_0$ is smooth, Ramanathan defined a notion of semistability for $G$-bundles on $C_0$ and constructed a projective moduli space $\mf{M}_{G}(C_0)$ of semistable $G$-bundles \cite{ramaphd}. These moduli spaces fit together into a flat, projective family $\mf{M}_G\to\cM_g$ over the stack of smooth genus $g$ curves. But, the moduli space of $G$-bundles over a singular curve is no longer projective, just as in the case of vector bundles, where one needs to include torsion-free sheaves in order to obtain compactifications. 

The problem of compactifying moduli of vector bundles on singular curves has a long and rich history and has been studied extensively, for example in  \cite{seshadrivb}, \cite{nagarajseshadriI}, \cite{nagarajseshadriII}, or \cite{pandharipande} (see \cite{belgib}, section 11, for a brief historical account). On the other hand, there is no agreed-upon compactification of the moduli of $G$-bundles for singular curves, although there are several out there. For example, Sun constructed a compactification for $G=\SL(r)$ in \cite{sun1} using certain classes of torsion-free sheaves, and in \cite{faltings} Faltings constructed compactifications for orthogonal and symplectic groups using torsion-free sheaves with a symmetric or alternating bilinear form. 

Ultimately, one would like to construct, for any $G$, a stack $\bbar{\mf{M}}_G\to\bbar{\cM}_g$ over the stack of stable curves that satisfies:
\begin{enumerate}
\item the morphism $\bbar{\mf{M}}_G\to\bbar{\cM}_g$ is flat and projective;
\item $\bbar{\mf{M}}_G$ corepresents a moduli functor that generalizes the notion of semistable $G$-bundles;
\item the moduli space of semistable $G$-bundles forms a dense open substack of $\bbar{\mf{M}}_G$.
\end{enumerate}
So far, no such compactification is known, but there are several very interesting and useful compactifications at hand.

In this paper, we will focus on one particular compactification, which is the moduli space of ``singular $G$-bundles" introduced by Schmitt (\cite{schspgb}, \cite{schbh}). Fixing a faithful representation $G\subset\GL(V)$, a singular $G$-bundle on $C_0$ consists of a uniform rank $r=\dim V$ torsion-free sheaf $\cE$ and a section
$$\hhat{\tau}:C_0\to\ul{\Hom}(\cE,V\otimes\cO_{C_0})\sslash G,$$
which should be thought of as a degenerate version of a reduction of structure group from $\GL(V)$ to $G$. Every $G$-bundle $E$ gives a singular $G$-bundle $(\cE,\hhat{\tau})$, where $\cE=E\times^GV$ and $\hhat{\tau}$ is the natural reduction of structure group. Conversely, if $\cE$ is a locally free sheaf and $\hhat{\tau}(C_0)\subset\ul{\Isom}(\cE,V\otimes\cO_{C_0})/G$, then we get a $G$-bundle as the pullback
\[\begin{tikzcd}
E\arrow{r}\arrow{d}&\ul{\Isom}(\cE,V\otimes\cO_{C_0})\arrow{d}\\
C_0\arrow{r}{\hhat{\tau}}&\ul{\Isom}(\cE,V\otimes\cO_{C_0})/G.\end{tikzcd}\]
A singular $G$-bundle which gives a $G$-bundle over a dense open subset of $C_0$ is called ``honest." 

In \cite{schmc}, Schmitt and Mu\~{n}oz-Casta\~{n}eda defined a notion of semistability for honest singular $G$-bundles and proved the existence of a projective moduli space of semistable honest singular $G$-bundles. Moreover, they showed that if $V$ has a $G$-invariant, nondegenerate bilinear form, then an honest singular $G$-bundle gives a $G$-bundle over the whole smooth locus of $C_0$ (loc. cit. theorem 3.5). There is a relative moduli space of singular $G$-bundles over $\bbar{\cM}_g$ constructed in \cite{mcuniformity}, and, although it does not necessarily meet the criteria for a compactification we described above, it serves as a very useful starting point.

However, there is another very natural candidate for a compactification, which comes from conformal blocks and is of geometric interest, but lacks a modular interpretation. For a $G$-module $V$, let $\cD(V)$ be the line bundle on the stack of $G$-bundles $\Bun_G(C_0)$, whose fiber over a $G$-bundle $E$ is the determinant of cohomology
$$\det\mH^0(C_0,E\times^GV)^*\otimes\det\mH^1(C_0,E\times^GV)$$
of the vector bundle $E\times^GV$. Let $A(C_0)$ be the section ring
$$A(C_0)=\bigoplus_{m\geq 0}\mH^0(\Bun_G(C_0),\cD(V)^m),$$
which has connections to the algebra of conformal blocks (\cite{belgib}, section 9). Specifically, if $d_V$ is the Dynkin index of $V$ (see \cite{knr}) and $\mf{g}$ the Lie algebra of $G$, then we have that
\begin{equation}\label{veronese}A(C_0)=\bigoplus_{m\geq 0}\mbb{V}_{\mf{g},md_V}^*\end{equation}
is the $d_V$-th Veronese subalgebra of the conformal blocks algebra. By the work of several authors (\cite{faltingsverlinde}, \cite{beauville-laszlo}, \cite{knr}, \cite{laszlo-sorger}), when $C$ is a smooth curve there is an ample line bundle $\Theta(V)$ on Ramanathan's moduli space $\mf{M}_G(C)$ and an algebra isomorphism
$$\bigoplus_{m\geq 0}\mH^0(\mf{M}_G(C),\Theta(V)^m)\cong\bigoplus_{m\geq 0}\mH^0(\Bun_G(C),\cD(V)^m).$$
In particular $\mf{M}_G(C)\cong\Proj A(C)$. 

Letting $\cA=\bigoplus_{m\geq 0} p_*\cD(V)^m$ for the relative stack of $G$-bundles $p:\Bun_G\to\bbar{\cM}_g$, the stack $\Proj\cA$ seems like quite a nice compactification of $\mf{M}_G$ -- it is flat over $\bbar{\cM}_g$ with normal integral fibers and agrees with Ramanathan's moduli space over the locus of smooth curves $\cM_g$. However, we have to show that $\cA$ is a finitely generated algebra in order to consider $\Proj\cA$ a compactification. Finite generation was proved by Belkale-Gibney in the case $G=\SL(r)$ (\cite{belgib}), and Moon-Yoo generalized this to parabolic $\SL(r)$-bundles on pointed curves in \cite{moonyoo}. Note that one does not necessarily expect finite generation for these algebras -- the fact that the spaces $\mH^0(\Bun_G(C_0),\cD(V)^m)$ are even finite-dimensional is already surprising, since we only know this through the connection to conformal blocks (finite-dimensionality of conformal blocks is proven in \cite{tuy}). 

Our aim is thus two-fold: a) prove finite generation of $\cA$ for other groups $G$, and b) relate $\cA$ to the moduli of singular $G$-bundles in hope of finding a modular interpretation for $\Proj\cA$. So far, we have succeeded only in the case of type A or C Lie groups. Our main results for these groups are as follows.

\begin{theorem}\label{mainthm} Let $V$ be a finite-dimensional representation of a simple, simply connected Lie group of type A or C. Then for any stable curve $C_0$ of genus $g\geq 2$ the algebra
\begin{equation}\label{mainthmeqtn}A(C_0)=\bigoplus_{m\geq 0}\mH^0(\Bun_G(C_0),\cD(V)^m)\end{equation}
is finitely generated.\end{theorem}

\noindent Our proof is based on Belkale-Gibney's, the main point being to show that $A(C_0)$ is isomorphic (in sufficiently high degree) to the section ring of an ample line bundle on a normalized moduli space of singular $G$-bundles (theorem \ref{isodkl}). Technically this isomorphism is only for certain choices of representation $V$, but the choice of representation does not matter for finite generation by equation (\ref{veronese}). The crux of the proof is a pole calculation that shows sections on $\Bun_G(C_0)$ extend to a (suitably normalized) stack of singular $G$-bundles. We show that there is no pole in type A or C, but there may be for type B, D (example \ref{polebd}). However, this pole is bounded by a number that scales linearly with $m$ (the $m$ appearing in (\ref{mainthmeqtn})), and we believe this should allow theorem \ref{mainthm} to generalize to other groups (see the discussion at the end of example \ref{polebd} -- one still needs to show the new line bundle correcting the pole descends to an ample line bundle on the appropriate moduli space).

In section \ref{conformalblocks}, we explain how the proof of theorem \ref{mainthm} implies finite generation of the sheaf of algebras $\cA$ on $\bbar{\cM}_g$, and taking $\mf{X}=\Proj\cA$ we obtain the following.

\begin{theorem}\label{flatfamily} Let $G$ be a simple, simply connected Lie group of type A or C, and let $g\geq 2$. Then there is a flat, relatively projective family $\mf{X}\to\bbar{\cM}_g$ such that
\begin{enumerate}
\item the fiber over a smooth curve is Ramanathan's moduli space of semistable $G$-bundles;
\item the fiber over an arbitrary curve is a normalized moduli space of semistable honest singular $G$-bundles (for any choice of representation which is as in section \ref{sectnotation}).
\end{enumerate}\end{theorem}

\noindent Please note that, in item 2., the moduli space we deal with is possibly a bit smaller than usual, as we are only able to include the singular $G$-bundles which are in the closure of $\Bun_G(C_0)$. It seems important to determine whether $\Bun_G(C_0)$ is dense in the stack of honest singular $G$-bundles (for appropriate choices of representation), but we are so far unable to do so. Also we should clarify that our moduli spaces are not explicitly the normalizations of the usual moduli spaces of singular $G$-bundles, but rather we normalize the ``master space" of which the usual moduli space is a GIT quotient. See section \ref{setupforfg} for the exact definition of the moduli spaces.

Since normalizations do not behave well with base change, item 2. of theorem \ref{flatfamily} only applies to fibers over fixed stable curves, i.e. we are not able to give a modular interpretation of $\mf{X}=\Proj\cA$ over a \emph{family} of stable curves. This would be an interesting problem to resolve in future work. There are other approaches to compactification that perhaps one could use, for example in \cite{balaji} Balaji recently constructed a Gieseker-type moduli space over $\Spec\CC[[t]]$ that provides a relative compactification compatible with a degeneration of a smooth curve into an irreducible curve with one node (see also \cite{martens-thaddeus}, \cite{psolis}). It would be interesting to see how his moduli spaces relate to conformal blocks. It is also possible that $\Proj\cA$ has no true modular interpretation, for example the Satake compactification $A_g^*$ of the moduli space $A_g$ of principally polarized abelian varieties of dimension $g\geq 2$ (\cite{baily}, \cite{faltingschai}) has no known modular interpretation. However, $A_g^*$ is known to be a ``minimal" compactification, in the sense that any other smooth toroidal compactification maps canonically to $A_g^*$ (\cite{faltingschai}, theorem 2.3). We wonder if the conformal blocks compactification $\Proj\cA$ has such a property.

\subsection{Organization of the paper}

Section \ref{background} is background on torsion-free sheaves, singular $G$-bundles, and their semistability. We show that the definition of semistability is independent of the choice of polarization (proposition \ref{lind}).

The rest of the paper is divided into three parts. Section \ref{surjection} is done at the stack level and is devoted to showing that sections of $\cD(V)$ over $\Bun_G(C_0)$ extend to the normalization of the stack of honest singular $G$-bundles (technically, we only deal with the closure of $\Bun_G(C_0)$ in this stack). For this, we have to show that sections on $\Bun_G(C_0)$ extend over a one-parameter family of $G$-bundles degenerating into an honest singular $G$-bundle. We show that any such family can be lifted to a family of descending $G$-bundles on the normalization of $C_0$ (sects. \ref{bhbundles}, \ref{liftsubsection}), and using a factorization of $\mH^0(\Bun_G(C_0),\cD^l)$ (lemma \ref{lambdadecomp}) we are able to show that sections extend by way of an explicit pole calculation (sect. \ref{polesubsection}).

Sections \ref{part2} and \ref{proofmainthm} deal with the actual moduli spaces and go through the geometric invariant theory setup necessary to prove theorem \ref{mainthm}. We outline the construction of a polarized, normalized moduli space $(\cX,L)$ of semistable honest singular $G$-bundles (sects. \ref{paramspaces}-\ref{setupforfg}) due to Schmitt and Mu\~{n}oz-Casta\~{n}eda \cite{schmc}. After identifying $L$ (sect. \ref{linebundleidentities}), we establish an injection $\mH^0(\cX,L)\hookrightarrow\mH^0(\Bun_G(C_0),\cD^l)$ (section \ref{injsections}), which is an isomorphism by the above section extension property.

In the last section (sect. \ref{conformalblocks}), we show how this implies finite generation of $\cA$ and prove theorem \ref{flatfamily}, as well as discuss the connection to conformal blocks.\\

\noindent\tbf{Acknowledgements.} Thanks to P. Belkale, N. Fakhruddin, A. Gibney, S. Kumar, A. Mu\~{n}oz-Casta\~{n}eda, and A. Schmitt for useful comments and suggestions.

\section{Background}\label{background}

\subsection{Preliminaries}

A coherent sheaf $\cE$ on a Noetherian scheme $X$ is \emph{torsion-free} if every nonzero coherent subsheaf $\cE'\subseteq\cE$ is supported in dimension $\dim X$. The \emph{torsion subsheaf} $T(\cE)\subseteq\cE$ is the maximal subsheaf supported in dimension $<\dim X$ (this exists and is coherent), and $\cE$ is torsion-free if and only if $T(\cE)=0$. If $X$ is a reduced curve over an algebraically closed field, then $\cE$ is torsion-free if and only if it has one-dimensional support and has depth one at every closed point of its support.

Suppose $X$ is projective over a field, and $\cL$ is an ample line bundle on $X$. The Hilbert polynomial of a coherent sheaf $\cE$ is the polynomial
$$\mP_\cE(n)=\chi(\cE\otimes\cL^n).$$
We may express
$$\mP_\cE(n)=\sum_{k=1}^{\dim\cE}a_k(\cE)\frac{n^k}{k!}$$
for some coefficients $a_k(\cE)\in\ZZ$, and the rank and degree of $\cE$ are defined as
$$\rk\cE=\frac{a_d(\cE)}{a_d(\cO_X)},\;\;\;\;\;\;\;\deg\cE=a_{d-1}(\cE)-\rk\cE\cdot a_{d-1}(\cO_X),$$
where $d=\dim X$ (take $a_k(\cE)=0$ for $k>\dim\cE$). We will say that $\cE$ has \emph{uniform rank $r$} if its restriction to every component of $X$ is rank $r$.

\subsection{Singular $G$-bundles}\label{spbintro}

Let $X$ be a projective variety over $\CC$, $G$ a reductive algebraic group, and $V$ a rank $r$ representation of $G$.

\begin{definition}\label{sgb} A singular $G$-bundle on $X$ is a pair $(\cE,\tau)$ consisting of a uniform rank $r$ torsion-free sheaf $\cE$ and a nontrivial algebra homomorphism
$$\tau:\Sym^*(V\otimes\cE)^G\to\cO_X.$$
Note that $\tau$ is given by a section 
$$\hhat{\tau}:X\to\ul{\Hom}(\cE,V^*\otimes\cO_X)\sslash G:=\ul{\Spec}\Sym^*(V\otimes\cE)^G.$$\end{definition}

Every $G$-bundle $E\to X$ gives a singular $G$-bundle $(\cE,\tau)$, where $\cE=E\times^GV$ and $\hhat{\tau}$ is the natural reduction of structure group
$$X=E/G\to (E\times^G\GL(V))/G=\ul{\Isom}(\cE,V^*\otimes\cO_X)/G.$$
Conversely, given a singular $G$-bundle $(\cE,\tau)$ such that $\cE$ is locally free and $\hhat{\tau}(X)\subseteq\ul{\Isom}(\cE,V^*\otimes\cO_X)/G$, we get back a $G$-bundle $E\to X$ as the pullback
\[\begin{tikzcd}
E\arrow{r}\arrow{d}&\ul{\Isom}(\cE,V^*\otimes\cO_X)\arrow{d}\\
X\arrow{r}{\hhat{\tau}}&\ul{\Isom}(\cE,V^*\otimes\cO_X)/G.\end{tikzcd}\]
Hence, we will use the term ``$G$-bundle" interchangeably for the singular $G$-bundles which give $G$-bundles under the above construction. It will also be important for us to distinguish a certain class of ``honest" singular $G$-bundles, which are very close to being $G$-bundles.

\begin{definition}\label{honestsgb} A singular $G$-bundle $(\cE,\tau)$ is called an honest singular $G$-bundle if there is a dense open subset $U\subseteq X$ with $\cE_U$ locally free and $\hhat{\tau}(U)\subseteq\ul{\Isom}(\cE_U,V^*\otimes\cO_U)/G$.\end{definition}

\noindent Thus, an honest singular $G$-bundle is one which gives a $G$-bundle over a dense open subset of $X$.

\subsection{Semistable singular $G$-bundles on nodal curves}\label{sbonnodalcurves}

Now suppose $X$ is a connected (possibly reducible) nodal curve with ample line bundle $\cL$, and $G$ is semisimple. In \cite{schmc}, Schmitt and Mu\~{n}oz-Casta\~{n}eda defined a notion of semistability for honest singular $G$-bundles on nodal curves and showed that there is a projective moduli space of semistable honest singular $G$-bundles. To give the definition of semistability, we first define a reduction of a singular $G$-bundle to a tuple of one-parameter subgroups. Let $\eta_i=\Spec K_i$, $1\leq i\leq t$, be the generic points of the irreducible components of $X$. 

\begin{definition}\label{redlambda}\emph{(\cite{schmc}, section 3)} Let $(\cE,\tau)$ be an honest singular $G$-bundle and $E\to U$ the induced $G$-bundle over a dense open subset $U\subseteq X$. For a tuple $\vec{\lambda}=(\lambda^1,\dots,\lambda^t)$ of one-parameter subgroups of $G$, define a reduction of $(\cE,\tau)$ to $\vec{\lambda}$ to be a tuple $\vec{s}=(s_1,\dots,s_t)$ of points $s_i\in E_{\eta_i}/P(\lambda^i)_{K_i},$ where 
$$P(\lambda^i)=\{g\in G:\lim\limits_{t\to 0}\lambda^i(t)g\lambda^i(t)^{-1}\text{ exists in }G\}$$ 
is the parabolic subgroup associated to $\lambda^i$.\end{definition}

The semistability condition involves weighted filtrations associated to each reduction of the singular $G$-bundle. A weighted filtration of a sheaf $\cE$ is a pair $(\cE_\bullet,m_\bullet)$ consisting of a filtration by subsheaves
$$0=\cE_0\subset\cE_1\subset\cdots\subset\cE_{q+1}=\cE$$
and a sequence of rational numbers $m_\bullet=(m_1,\dots,m_q).$ Given an honest singular $G$-bundle $(\cE,\tau)$ and a reduction $\vec{s}$ of $(\cE,\tau)$ to $\vec{\lambda}$, we can form a weighted filtration $(\cE_\bullet,m_\bullet)$ as follows. Let $\ttilde{P}(\lambda^i)\subseteq\GL(V)$ be the parabolic subgroup for $\lambda^i$ in $\GL(V)$, and let $\lambda^i_1<\cdots<\lambda^i_{q_i+1}$ be the distinct weights of $\lambda^i$ acting on $V^*$. Via the embedding 
$$E/P(\lambda^i)\hookrightarrow\Isom(\cE_U,V^*\otimes\cO_U)/\ttilde{P}(\lambda^i),$$
the sections $s_i$ give partial flags
$$0=W^i_0\subset W^i_1\subset\cdots\subset W^i_{q_i+1}=W^i$$
in $W^i=\cE_{\eta_i}$ for $1\leq i\leq t$. So we have a weighted flag at each fiber $\cE_{\eta_i}$, and we take the ``direct sum" of these weighted flags to get a weighted flag in $W=\cE_{\eta_1}\oplus\cdots\oplus\cE_{\eta_t}$. This is done as follows. Let
$$\mu_1<\cdots<\mu_{q+1}$$
be the distinct values of the $\lambda^i_j$, and for each $1\leq i\leq t$ and $1\leq j\leq q+1$ define $W_j(i)=W^i_k$, where $k$ is maximal such that $\lambda^i_k\leq\mu_j.$ Then we have a filtration of $W$ given by
$$0=W_0\subset W_1\subset\cdots\subset W_{q+1}=W,$$
where $W_j=W_j(1)\oplus\cdots\oplus W_j(t).$ Let $\bbar{\cE}_j=\cE/(\cE\cap\iota_*W_j)$, and define the filtration $0\subset\cE_1\subset\cdots\subset\cE_{q+1}=\cE$ by
$$\cE_j=\ker[\cE\to\bbar{\cE}_j/T(\bbar{\cE}_j)].$$
The weights $m_\bullet=(m_1,\dots,m_q)$ are defined by $m_j=(\mu_{j+1}-\mu_j)/r.$

\begin{definition}\label{honestss} \emph{(\cite{schmc}, section 3)} A singular $G$-bundle $(\cE,\tau)$ is semistable if it is honest and, for every nontrivial tuple of one-parameter subgroups $\vec{\lambda}$ (i.e. not all constant) and reduction $\vec{s}$ of $(\cE,\tau)$ to $\vec{\lambda}$, we have
\begin{equation}\label{honestsseqtn}\sum_{j=1}^qm_j(\chi(\cE)\rk(\cE_j)-\chi(\cE_j)\rk(\cE))\geq 0,\end{equation}
where $(\cE_\bullet,m_\bullet)$ is the weighted filtration constructed from $(\vec{\lambda},\vec{s})$ as above. We say $(\cE,\tau)$ is stable if strict inequality holds for every such $(\vec{\lambda},\vec{s})$.\end{definition}

A somewhat surprising fact about this type of semistability is that it does not depend on the choice of polarization.

\begin{proposition}\label{lind} Semistability for honest singular $G$-bundles does not depend on the choice of polarization $\cL$.\end{proposition}
\begin{proof} Note that the second term $\chi(\cE_j)\rk(\cE)$ appearing in (\ref{honestsseqtn}) does not depend on the polarization, because $\cE$ is uniform rank, and because Euler characteristics never depend on polarization. Thus, looking only at the first term $\chi(\cE)\rk(\cE_j)$, it suffices to show that
\begin{equation}\label{sum1stterm}\sum_{j=1}^qm_j\rk(\cE_j)\end{equation}
is independent of polarization.

Let $X_1,\dots,X_t$ be the irreducible components of $X$, and let $1\leq j_1<\cdots<j_{q_{i+1}}\leq q+1$ be the indices such that $\mu_{j_k}=\lambda^i_k$. Then
$$\rk\cE_j|_{X_i}=\dim W_j(i)=\dim W^i_k$$
for $j_k\leq j<j_{k+1}$. Hence for each component $X_i$ we have
$$\sum_{j=1}^qm_j\rk(\cE_j|_{X_i})=\frac{1}{r}\sum_{j=1}^q(\mu_{j+1}-\mu_j)\dim W_j(i).$$
Since the ranks $\dim W_j(i)$ are the same for the summands indexed by $j_k\leq j<j_{k+1}$, the sum breaks into several telescoping sums and simplifies to
$$\frac{1}{r}\left(\sum_{k=1}^{q_i}(\lambda^i_{k+1}-\lambda^i_k)\dim W^i_k+r(\mu_{q+1}-\lambda^i_{q_i+1})\right)=\mu_{q+1}-\frac{1}{r}\sum_{k=1}^{q_i+1}\lambda^i_k(\dim W^i_k-\dim W^i_{k-1}),$$
which is just equal to $\mu_{q+1}$, because $\lambda^i:\bG_m\to G\to\GL(V)$ lands in $\SL(V)$ if $G$ is semisimple.

Thus, we have shown that
\begin{equation}\label{compsum}\sum_{j=1}^qm_j(\rk\cE_j|_{X_i})=\mu_{q+1}\end{equation}
for each component $X_i$. To compute (\ref{sum1stterm}) in terms of (\ref{compsum}), we will use the following rank formula. For any torsion-free sheaf $\cF$ on the nodal curve $X$, there is an exact sequence (\cite{seshadrivb}, section 7.1)
$$0\to\cF\to\bigoplus_{i=1}^t\cF|_{X_i}\to\cT\to 0$$
with $\cT$ a torsion sheaf (supported only at points where two components meet). It follows that
\begin{equation}\label{rankformula}\rk\cF=\frac{1}{\deg\cL}\sum_{i=1}^td_i\rk\cF|_{X_i},\end{equation}
where $d_i=\deg\cL|_{X_i}$. Combining (\ref{compsum}) and (\ref{rankformula}) we get
\begin{equation}\label{sumconclusion}\sum_{j=1}^qm_j\rk(\cE_j)=\mu_{q+1}\end{equation}
as well, which proves the proposition.\end{proof}

\begin{remark} The semistability inequality (\ref{honestsseqtn}) can therefore also be written as
\begin{equation}\label{ssalt1}\frac{\sum_{j=1}^qm_j\chi(\cE_j)}{\mu_{q+1}}\leq\frac{\chi(\cE)}{r},\end{equation}
or alternatively
\begin{equation}\label{ssalt2}\sum_{j=1}^{q+1}\mu_j\chi(\cE_j/\cE_{j-1})\geq 0.\end{equation} 
It is interesting to compare this, in the case $G=\SL(r)$, to the usual definition of semistability for sheaves, which requires
\begin{equation}\label{mumfss}\frac{ \chi(\cE')}{r'}\leq\frac{\chi(\cE)}{r}\end{equation}
for every nonzero subsheaf $\cE'\subseteq\cE$ with $r'=\rk\cE'$. If $G=\SL(r)$ and $X$ is irreducible, then semistability of a singular $G$-bundle is equivalent to semistability of the underlying sheaf, so we can think of definition \ref{honestss} as a generalization of sheaf semistability (perhaps it seems most natural when we write both notions of semistability in the form (\ref{ssalt2})).\end{remark}

\subsection{Good choices of representation}\label{vsubsect}

We continue to assume that $X$ is a projective, connected nodal curve, and $G$ a semisimple group with a faithful representation $G\subset\GL(V)$. For our purposes later on (e.g. prop. \ref{lift}), it will be important that the singular $G$-bundles we deal with are not just honest, but even give $G$-bundles over the entire smooth locus of $X$. Schmitt and Mu\'{n}oz-Casta\'{n}eda showed that every honest singular $G$-bundle will have this property if the representation $V$ is chosen appropriately.

\begin{proposition}\label{goodv}\emph{(\cite{schmc}, theorem 3.5)} Assume that $V$ has a $G$-invariant, nondegenerate bilinear form.\footnote{Every group has such a representation, because if $V$ is any representation, then $V\oplus V^*$ has an invariant, nondegenerate bilinear form.} Then any honest singular $G$-bundle $(\cE,\tau)$ with $\deg\cE=0$ gives a $G$-bundle over the smooth locus of $X$.\end{proposition}
\begin{proof}Let us sketch the proof from \cite{schmc}. We first construct a bilinear form $\cE\otimes\cE\to\cO_{C_0}$ as follows. Consider the morphisms
\[\begin{tikzcd}
\ul{\Hom}(\cE,V^*\otimes\cO_{C_0})\arrow{rr}{q}\arrow{dr}{p}&&\ul{\Hom}(\cE,V^*\otimes\cO_{C_0})\sslash G\arrow{dl}{\bbar{p}}\\
&C_0&
\end{tikzcd}\]
Over $\mbf{H}:=\ul{\Hom}(\cE,V^*\otimes\cO_{C_0})=\ul{\Spec}\Sym^*(V\otimes\cE)$ there is a universal map
$$p^*\Sym^*(V\otimes\cE)\to\cO_{\mbf{H}},$$
whose degree one part gives a map
$$p^*\cE\to V^*\otimes\cO_{\mbf{H}}.$$
Let $f:V\to V^*$ be the isomorphism induced by the bilinear form on $V$, and $B^\dagger:V^*\otimes V^*\to\CC$ the bilinear form given by $f^{-1}$. Then we have a bilinear form on $p^*\cE$,
$$\psi:p^*\cE\otimes p^*\cE\to V^*\otimes V^*\otimes\cO_{\mbf{H}}\xrightarrow{B^\dagger}\cO_{\mbf{H}}.$$
Note that the universal map $p^*\Sym^*(V\otimes\cE)\to\cO_{\mbf{H}}$ is $G$-equivariant, hence so is $\psi$ by $G$-invariance of $B^\dagger$. Since 
$$p^*\cE\otimes p^*\cE\cong p^*(\cE\otimes\cE)=q^*\bbar{p}^*(\cE\otimes\cE),$$
we get by adjunction a map
$$\bbar{\psi}:\bbar{p}^*(\cE\otimes\cE)\to q_*\cO_{\mbf{H}}.$$
By the equivariance, the image lies in $(q_*\cO_\mbf{H})^G=\cO_{\mbf{H}\sslash G}$. Pulling back $\bbar{\psi}$ by the section $\hhat{\tau}:C_0\to\mbf{H}\sslash G$ gives a bilinear form $\varphi:\cE\otimes\cE\to\cO_{C_0}.$

It easy to check that, for any smooth point $x\in X$, the map $\cE|_x\to\cE|_x^*$ induced by $\varphi$ factorizes as
\begin{equation}\label{formfactorization}\cE|_x\xrightarrow{\hhat{\tau}_x}V^*\xrightarrow{f^{-1}}V\xrightarrow{\hhat{\tau}_x^*}\cE|_x^*,\end{equation}
where $\hhat{\tau}_x\in\Hom(\cE|_x,V^*)$ denotes any preimage of $\hhat{\tau}(x)\in\Hom(\cE|_x,V^*)\sslash G$. In particular, $\cE\to\cE^\vee$ is surjective (on stalks, not just fibers) at any smooth point where $\hhat{\tau}(x)\in\Isom(\cE|_x,V^*)/G$. Since $\cE$ and $\cE^\vee$ have the same multirank, this implies that $\cE\to\cE^\vee$ is an isomorphism over the maximal open subset of the smooth locus where $(\cE,\tau)$ gives a $G$-bundle. Then $\cE\to\cE^\vee$ is injective over a dense open subset, so it is injective everywhere because $\cE$ is torsion-free. By \cite{schmc}, appendix, we have $\chi(\cE)=\chi(\cE^\vee)$ for any degree zero torsion-free sheaf on a nodal curve, so $\cE\to\cE^\vee$ is an isomorphism, and the factorization (\ref{formfactorization}) implies that $\hhat{\tau}(x)\in\Isom(\cE|_x,V^*)/G$ for every smooth point $x\in X$. Thus $(\cE,\tau)$ is a $G$-bundle over the entire smooth locus.\end{proof}

\section{Notation}\label{sectnotation}

The following notation will be fixed for the remainder of the chapter. Let $C_0$ be a stable curve of genus $g\geq 2$ with normalization $\nu:C\to C_0$. Let $S$ be the set of nodes of $C_0$, and for $x\in S$ let $\nu^{-1}(x)=\{x_1,x_2\}.$ Let $G$ be a simple, simply-connected algebraic group, and fix a faithful representation $G\subset\GL(V)$ of rank $r$. We will assume that $V$ has a nondegenerate bilinear form preserved by $G$, so that by proposition \ref{goodv} every degree zero honest singular $G$-bundle on $C_0$ gives a $G$-bundle on $C_0-S$. In particular this is true of any honest singular $G$-bundle which is a flat limit of $G$-bundles.\footnote{Indeed, proposition \ref{goodv} applies, because degree is constant in flat families.} For a singular $G$-bundle $(\cE,\tau)$ and point $x\in C_0$, we denote by $\hhat{\tau}_x\in\Hom(\cE|_x,V^*)$ any preimage of $\hhat{\tau}(x)\in\Hom(\cE|_x,V^*)\sslash G$.

\section{Section extension problem}\label{surjection}

\subsection{Overview}\label{partioverview}

We will work with the following stacks.
\begin{itemize}
\item $\Bun_G(C_0)$ the stack of $G$-bundles, whose fiber over a scheme $T$ is the groupoid of $G$-bundles on $C_0\times T$.
\item $\SB_G(C_0)$ the stack of singular $G$-bundles, whose fiber over a scheme $T$ is the groupoid of singular $G$-bundles on $C_0\times T$ (we require that the underlying sheaf is flat over $T$ and its restriction to each fiber over $T$ is torsion-free, uniform rank $r$).
\item $\SB_G^*(C_0)$ the open substack of honest singular $G$-bundles (meaning fiberwise honest over the base).
\item $\SB_G^0(C_0)$ the closure of $\Bun_G(C_0)$ in $\SB_G^*(C_0)$.
\end{itemize}
All are algebraic stacks locally of finite type over $\CC$. $\Bun_G(C_0)$ is smooth and connected (\cite{belfakh}, proposition 5.1) and forms a dense open substack of $\SB_G^0(C_0)$. 

Let $\cD$ be the determinant of cohomology line bundle on the stack of coherent sheaves $\Coh(C_0)$, whose fiber over a sheaf $\cE$ is
$$\det\mH^0(C_0,\cE)^*\otimes\det\mH^1(C_0,\cE).$$
Let the pullback of $\cD$ to $\Bun_G(C_0)$ along the contraction map
$$\Bun_G(C_0)\to\Coh(C_0),$$
$$E\mapsto E\times^GV$$
also be denoted $\cD$.

The stack $\SB_G(C_0)$ similarly carries a determinant of cohology line bundle, and the goal of this section is to show that sections of $\cD$ over $\Bun_G(C_0)$ extend to the normalization of $\SB_G^0(C_0)$. Recall that an algebraic stack $\cS$ is normal if there is a smooth surjection $U\to\cS$ with $U$ a normal scheme. Any locally Noetherian algebraic stack $\cS$ has a normalization $\ttilde{\cS}\to\cS$, which is defined by the property that $\ttilde{\cS}\to\cS$ is representable and, for any scheme $T$ and smooth morphism $T\to\cS$, the base change $\ttilde{\cS}\times_\cS T\to T$ is the normalization of $T$.

We will prove:

\begin{theorem}\label{sectionextensiontheorem} Suppose $G$ is a simple Lie group of type A or C, and the representation $V$ is chosen as follows:
\begin{enumerate}[label=(\roman*)]
\item if $G=\SL(n)$, choose $V=W\oplus W^*$, where $W=\CC^n$ is the standard representation;
\item if $G=\Sp(2n)$, choose $V=\CC^{2n}$ the standard representation.
\end{enumerate}
Let $\cY$ be the normalization of $\SB_G^0(C_0)$, and let $\cD$ denote the pullback of $\cD$ to $\cY$.Then for any $l\geq 0$, the restriction map
$$\mH^0(\cY,\cD^l)\to\mH^0(\Bun_G(C_0),\cD^l)$$
is an isomorphism.\end{theorem}
\noindent Note that sections can be restricted from $\cY$ to $\Bun_G(C_0)$, because $\Bun_G(C_0)$ is smooth and therefore forms an open substack of $\cY$.

\subsection{Brief outline of proof of theorem \ref{sectionextensiontheorem}}\label{pfoutlinepole}

The plan for the proof of theorem \ref{sectionextensiontheorem} is as follows. For $k$ a characteristic zero field, let $A=k[[t]]$, $K=k((t))$. We have to show that, for any map 
$$f:\Spec A\to\SB_G^0(C_0)$$ 
sending $\Spec K$ into $\Bun_G(C_0)$, the pullback of a section on $\Bun_G(C_0)$ has no pole at $t=0$. We first recall the definition of descending $G$-bundles from \cite{schspgb} and show that any map $f$ as above lifts to the stack of descending $G$-bundles (sect. \ref{liftsubsection}). This allows us to reduce to doing the pole calculation in the stack of descending $G$-bundles, as described in situation \ref{situation'}. Using a factorization of $\mH^0(\Bun_G(C_0),\cD^l)$ (lemma \ref{lambdadecomp}), we are then able to compute the pole explicitly (sect. \ref{polesubsection}).

\subsection{Descending $G$-bundles}\label{bhbundles}

Bhosle introduced (e.g. \cite{bhosle}) the following method to model torsion-free sheaves on $C_0$ in terms of certain bundles on the normalization $C$, which we call ``Bhosle bundles." A Bhosle bundle is a pair $(\cF,\vec{q})$ consisting of a rank $r$ vector bundle $\cF$ on $C$ and a collection $\vec{q}=(q_x)_{x\in S}$ of rank $r$ quotients
$$q_x:\cF|_{x_1}\oplus\cF|_{x_2}\to Q_x.$$ 
Every Bhosle bundle yields a torsion-free sheaf
$$\cE=\ker[\nu_*\cF\to\bigoplus_{x\in S}(i_x)_*Q_x]$$
on $C_0$, and every torsion-free sheaf on $C_0$ arises this way (see \cite{sun}, lemma 2.1, for the proof in the irreducible case; the reducible case is similar). If the coordinate maps $q_{x_1}$, $q_{x_2}$ are isomorphisms, then $-q_{x_2}^{-1}q_{x_1}$ gives an identification $\cF|_{x_1}\xrightarrow\sim\cF|_{x_2}$, and $\cE$ is the vector bundle resulting from gluing the fibers of $\cF$ along this identification. If one of the $q_{x_i}$ is not an isomorphism, then $\cE$ is not locally free, but the local structure of $\cE$ is still determined by the ranks of the $q_{x_i}$ (see loc. cit.).

The analagous model for $G$-bundles is called a ``descending $G$-bundle" (see \cite{schspgb}, section 4.3, for the definition of families).

\begin{definition}\label{sbb} A descending $G$-bundle is a triple $(\cF,\sigma,\vec{q})$, where $(\cF,\sigma)$ is a $G$-bundle on $C$ (presented as a singular $G$-bundle), and $(\cF,\vec{q})$ is a Bhosle bundle such that the image of
\begin{equation}\label{dgbdefn}\Sym^*(V\otimes\cE)^G\to\nu_*\Sym^*(V\otimes\cF)^G\xrightarrow{\nu_*\sigma}\nu_*\cO_C\end{equation}
is contained in $\cO_{C_0}\subset\nu_*\cO_C$, where $\cE$ is the induced torsion-free sheaf.\end{definition}

\begin{remark} Note that, even though $\nu_*$ does not commute with $\Sym$, there is a natural map 
$$\Sym^*(\nu_*(V\otimes\cF))\to\nu_*\Sym^*(V\otimes\cF),$$ 
 and (\ref{dgbdefn}) is obtained from the composition 
$$\Sym^*(V\otimes\cE)\to\Sym^*(\nu_*(V\otimes\cF))\to\nu_*\Sym^*(V\otimes\cF).$$\end{remark}

\begin{notation} For a descending $G$-bundle $(\cF,\sigma,\vec{q})$, we will write $q_{x_i}$ for the restriction of $q_x$ to $\cF|_{x_i}$ and $\kappa_{x_i}=\det q_{x_i}\circ(\det\hhat{\sigma}_{x_i})^{-1}$, $i=1,2$.\end{notation}

By design, every descending $G$-bundle $(\cF,\sigma,\vec{q})$ on $C$ gives a singular $G$-bundle $(\cE,\tau)$ on $C_0$. The assignment $(\cF,\sigma,\vec{q})\mapsto(\cE,\tau)$ has the following property (which can be generalized to families).

\begin{proposition}\label{bhisos} Suppose $(\cF,\sigma,\vec{q})$ and $(\cF',\sigma',\vec{q}')$ are descending $G$-bundles which produce actual $G$-bundles $(\cE,\tau)$ and $(\cE',\tau')$ on $C_0$. Given an isomorphism $(\cE,\tau)\xrightarrow\sim(\cE',\tau')$, there is a unique isomorphism $(\cF,\sigma,\vec{q})\xrightarrow\sim(\cF',\sigma',\vec{q}')$ making the following diagram commute:
\[
\begin{tikzcd}
0\arrow{r}&\cE\arrow{r}\arrow{d}&\nu_*\cF\arrow{r}\arrow{d}&\bigoplus_{x\in S}Q_x\arrow{r}\arrow{d}&0\\
0\arrow{r}&\cE'\arrow{r}&\nu_*\cF'\arrow{r}&\bigoplus_{x\in S}Q_x'\arrow{r}&0
\end{tikzcd}
\]
\end{proposition}
\begin{proof} The proof is straightforward, see e.g. \cite{mcthesis} section 4.3.\end{proof}

The next proposition gives a version of the descending $G$-bundle condition that is easier to use in practice (see the examples following the proposition). Recall the Schur functor $S^\lambda$ associated to an integer partition $\lambda=(\lambda_1\geq\cdots\geq\lambda_d\geq 0)$, which associates a subspace $S^\lambda(M)\subseteq M^{\otimes|\lambda|}$ to each finite-dimensional vector space $M$. A map of vector spaces $f:M\to M'$ induces a map $S^\lambda(f):S^\lambda(M)\to S^\lambda(M')$.

\begin{proposition}\label{dgb} A triple $(\cF,\sigma,\vec{q})$ consisting of a $G$-bundle $(\cF,\sigma)$ on $C$ and a Bhosle bundle $(\cF,\vec{q})$ is a descending $G$-bundle if and only if the map
\begin{equation}\label{descent}S^\lambda(\ker q_x)\xrightarrow{(S^\lambda(\pr_1),S^\lambda(\pr_2))}S^\lambda(\cF|_{x_1})\oplus S^\lambda(\cF|_{x_2})\xrightarrow{S^\lambda(\hhat{\sigma}_{x_1})-S^\lambda(\hhat{\sigma}_{x_2})}S^\lambda(V)^*\to(S^\lambda(V)^G)^*\end{equation}
is zero for every partition $\lambda$ and node $x\in S$. If all $q_{x_i}$ are isomorphisms, then this is equivalent to the requirement that the gluing functions
$$-\hhat{\sigma}_{x_2}q_{x_2}^{-1}q_{x_1}\hhat{\sigma}_{x_1}^{-1}:V^*\to V^*$$
lie in $G\subset\GL(V^*)$.\end{proposition}
\begin{proof} Recall that $\cO_{C_0}$ is the kernel of the map
$$\nu_*\cO_C\to(\nu_*\cO_C)|_x=\bigoplus_{x\in S}k(x_1)\oplus k(x_2)\to\bigoplus_{x\in S}k(x),$$
where the last map is given by $(a_1,a_2)\mapsto a_1-a_2$ in each summand. So, it suffices to check the descending $G$-bundle condition at the fiber over each node. The image of $\Sym^*(V\otimes\cE)$ in $\nu_*\Sym^*(V\otimes\cF)|_x=\Sym^*(V\otimes\cF|_{x_1})\oplus\Sym^*(V\otimes\cF|_{x_2})$ is the same as the image of
$$(\Sym^*(\pr_1),\Sym^*(\pr_2)):\Sym^*(V\otimes\ker q_x)\to\Sym^*(V\otimes\cF|_{x_1})\oplus\Sym^*(V\otimes\cF|_{x_2}),$$
so the requirement for descent is that the composition
\begin{equation}\label{descent2}\Sym^*(V\otimes\ker q_x)^G\to\Sym^*(V\otimes \cF|_{x_1})^G\oplus\Sym^*(V\otimes \cF|_{x_2})^G\to k(x_1)\oplus k(x_2)\to k(x)\end{equation}
is zero for all nodes $x\in S$, where the first and last map are as above and the middle map is $(\nu_*\sigma)|_x=(\sigma|_{x_1},\sigma|_{x_2}).$ The first part of the proposition then follows from the decomposition
$$\Sym^*(V\otimes M)\cong\bigoplus_{\lambda\vdash d,\atop{d\geq 0}}S^\lambda(V)\otimes S^\lambda(M).$$
for any vector space $M$.

For the second assertion, let
$$g=-\hhat{\sigma}_{x_2}q_{x_2}^{-1}q_{x_1}\hhat{\sigma}_{x_1}^{-1}\in\GL(V^*),$$
so that
$$\ker q=\{(v^*,gv^*):v^*\in V^*\}$$
as a subspace of $V^*\oplus V^*$ (using the $\hhat{\sigma}_{x_i}$ to identify $\cF|_{x_i}$ with $V^*$). By the first part of the proposition, the descent requirement becomes
$$w^*(gw)=w^*(w)$$
for all $w\in S^\lambda(V)^G$, $w^*\in S^\lambda(V)^*$. This is the same as saying $f(g)=f(e)$ for every $G$-invariant function $f$ on $\GL(V^*)$, or rather $\bbar{g}=\bbar{e}$ as points of $\GL(V^*)/G$ (note that $\GL(V^*)/G$ is affine since $G$ is reductive).\end{proof}

\begin{remark} The proposition is actually quite usable. As shown in the proof, the requirement for descent is that a certain algebra homomorphism
$$\Sym^*(V\otimes\ker q_x)^G\to\CC\oplus\CC$$
has image contained in the diagonal subalgebra of $\CC\oplus\CC$. So, it suffices to check (\ref{descent}) only in degrees $\lambda$ containing algebra generators for $\CC[\End(V)]^G=\CC[\End(V)\sslash G]$, where $G$ acts by $g\cdot f=f\circ g^{-1}$ on $\End(V)$. Proposition \ref{dgb} can then be used to show that, for the following classical groups $G\subset\GL(V)$, a descending $G$-bundle consists of a $G$-bundle $(\cF,\sigma)$ with a Bhosle structure $\vec{q}$ satisfying the following conditions.
\begin{enumerate}
\item For $G=\SL(V)$, we only need $\kappa_{x_1}=(-1)^r\kappa_{x_2}$.
\item If $V$ is a symplectic space and $G=\Sp(V)$, then the reduction of structure group $\hhat{\sigma}$ gives a symplectic form $\psi:\wedge^2\cF\to\cO_C$, and the descending $G$-bundle condition becomes: $\ker q_x$ is isotropic for the symplectic form
$$\langle(v_1,v_2),(w_1,w_2)\rangle=\psi_{x_1}(v_1,w_1)-\psi_{x_2}(v_2,w_2)$$
on $\cF|_{x_1}\oplus\cF|_{x_2}$.
\item If $V$ is a quadratic space and $G=\SO(V)$, then the condition is the same as for symplectic groups, with the addition that $\kappa_{x_1}=(-1)^r\kappa_{x_2}$.
\end{enumerate}
\end{remark}

\subsection{Lifting one-parameter families to the Bhosle stack}\label{liftsubsection}

Next we will show that any family of singular $G$-bundles given by a map $f$ as in section \ref{partioverview} lifts to a family of descending $G$-bundles.

\begin{proposition}\label{lift} Let $T$ be a smooth curve with a closed point $0\in T$, and let $(\cE,\tau)$ be a family of singular $G$-bundles on $C_0\times T$ that gives a $G$-bundle on the complement of $S\times 0$. Then $(\cE,\tau)$ is induced by a family of descending $G$-bundles $(\cF,\sigma,\vec{q})$ on $C\times T$.\end{proposition}

\begin{proof} Let $\cF=(\nu^*\cE)^{\vee\vee}$, and let $j:U\hookrightarrow C\times T$ be the inclusion of $U=C\times T-\nu^{-1}(S)\times T$. Note that $\cF$ is locally free by \cite{hart80}, corollary 1.4, and satisfies $j^*\cF\cong j^*\nu^*\cE$. Since $C\times T$ is normal, we have $j_*j^*\cO_{C\times T}\cong\cO_{C\times T}$, hence $j^*\nu^*\tau$ extends to an algebra homomorphism $\sigma:\Sym^*(V\otimes\cF)^G\to\cO_{C\times T}$. 

Since $\sigma$ is nondegenerate in codimension one (nondegenerate meaning $\hhat{\sigma}$ lands in $\ul{\Isom}(\cF,V^*\otimes\cO_C)/G$), it follows that $\sigma$ is nondegenerate everywhere, because the degeneracy locus is the divisor where $\det\hhat{\sigma}:\det\cF\to\cO$ vanishes. Thus, $(\cF,\sigma)$ is a $G$-bundle. To get the quotient maps $\vec{q}$, just take the ones we get from pulling back $\cE$ over $T-\{0\}$ and extend to $t=0$ using properness of Grassmannians (the condition on $\ker q_x$ given by proposition \ref{dgb} will continue to hold at $t=0$ since it is defined by the vanishing of a map of vector bundles on $T$).

Now, $(\cF,\sigma,\vec{q})$ induces a family of torsion-free singular $G$-bundles $(\cE',\tau')$ on $C_0$ which agrees with the original family over $U'=C_0\times T-S\times 0$. Both $\cE$ and $\cE'$ are flat families of depth 1 sheaves parametrized by a smooth curve, so they are S2 sheaves on $C_0\times T$ by \cite{EGA}, 6.3.1. Hence both equal their pushforward from $U'$. To see $\tau=\tau'$, note that the inclusion $\cO_{C_0\times T}\subseteq j_*\cO_{U'}$ gives an inclusion
$$\Hom_{C_0\times T}(\cM,\cO_{C_0\times T})\subseteq\Hom_{C_0\times T}(\cM,j_*\cO_{U'})=\Hom_{U'}(j^*\cM,\cO_{U'}).$$
for any sheaf $\cM$. Since $\tau=\tau'$ over $U'$, the two families coincide and the proposition is proved.\end{proof}

\subsection{Setup for pole calculation}

Recall that $A=k[[t]]$, $K=k((t))$. Let $\Bh_G(\nu)$ be the stack of descending $G$-bundles, and $\Bh_G^0(\nu)$ the open substack where all $q_{x_i}$ are isomorphisms. Let $\pi:\Bh_G(\nu)\to\SB_G^*(C_0)$ be the natural projection, which restricts to an isomorphism 
$$\Bh_G^0(\nu)\xrightarrow{\sim}\Bun_G(C_0).$$ 
By proposition \ref{lift}, a map $f:\Spec A\to\SB_G^0(C_0)$ as in section \ref{partioverview} lifts to a map $\ttilde{f}:\Spec A\to\Bh_G(\nu)$ sending $\Spec K$ into $\Bh^0_G(\nu)$. Thus, to prove theorem \ref{sectionextensiontheorem} it suffices to show the following.

\begin{situation}\label{situation'} For any map $\ttilde{f}:\Spec A\to\Bh_G(\nu)$ sending $\Spec K$ into $\Bh^0_G(\nu)$ and any section $s$ in $\mH^0(\Bh^0_G(\nu),\pi^*\cD^l)=\mH^0(\Bun_G(C_0),\cD^l)$, we must show $\ttilde{f}^*s$ has no pole at $t=0$.\end{situation}

In order to compute the pole of $\ttilde{f}^*s$, we will use the following factorization of $\mH^0(\Bun_G(C_0),\cD^l)$. Let $\scr{E}$ be the universal family of $G$-bundles parametrized by $\Bun_G(C)$, and for a dominant integral weight $\lambda$ of $G$ and point $x\in C$ let
$$\scr{E}_x^\lambda=\scr{E}_x\times^G V^\lambda,$$
where $V^\lambda$ is the irreducible representation of $G$ with highest weight $\lambda$.

\begin{lemma}\label{lambdadecomp}\emph{(\cite{belgib}, lemma 6.4)} Let $d_V$ be the Dynkin index of the $G$-module $V$ (see \cite{knr}). Then the pullback map $p:\Bun_G(C_0)\to\Bun_G(C)$ induces an isomorphism
\begin{equation}\label{lambdaeqn}\bigoplus_{\lambda}\mH^0(\Bun_G(C),\cD^l\otimes\bigotimes_{x\in S}\scr{E}_{x_1}^{\lambda_x}\otimes\scr{E}_{x_2}^{\lambda_x^*})\xrightarrow{\sim}\mH^0(\Bun_G(C_0),\cD^l),\end{equation}
where the sum is over all functions $\lambda$ assigning a dominant integral weight $\lambda_x$ of level $\leq ld_V$ to each node $x\in S$, and $\lambda_x^*$ denotes the highest weight of $(V^{\lambda_x})^*$.\end{lemma}

The isomorphism in the lemma has the following formula. Let $E\to C_0$ be a $G$-bundle and $s$ a section in the $\lambda$-component on $\Bun_G(C)$. Picking trivializations of $F=\nu^*E$ at $x_1$ and $x_2$ for each $x\in S$ gives a collection of transition elements $g_x\in G$, and we may express $s|_F$ as a sum of terms $\alpha\otimes\bigotimes_{x\in S}(v_x\otimes v_x^*)$ with $\alpha\in\cD(\cE)^l,$ $v_x\in V^{\lambda_x},$ $v_x^*\in V^{\lambda_x^*},$ where $\cE=E\times^GV$. Then as a section on $\Bun_G(C_0)$, $s|_E$ is the corresponding sum of the terms $(\prod_{x\in S}v_x^*(g_xv_x))\alpha.$

\subsection{Pole calculation}\label{polesubsection}

We will resolve situation \ref{situation'} by bounding the pole of $\ttilde{f}^*s$ in proposition \ref{pole}. Before carrying out the pole calculation, let me point out the following two items. First, if $(\cF,\sigma,\vec{q})$ is a family of descending $G$-bundles given by a map $\ttilde{f}$ as in situation \ref{situation'}, then by proposition \ref{dgb} the gluing function $g_x=-\hhat{\sigma}_{x_2}q_{x_2}^{-1}q_{x_1}\hhat{\sigma}_{x_1}^{-1}$ is in $G(K)$ for each node $x\in S$. But, the element $g_x\in G(K)$ is only well-defined up to the left and right action of $G(A)$, as it depends on a choice of $G(A)$-coset representatives of the $\hhat{\sigma}_{x_i}\in\Isom(\cF|_{x_i},V^*\otimes A)/G_A$. Recall that the double cosets $G(A)\backslash G(K)/G(A)$ are parametrized by dominant one-parameter subgroups of a maximal torus of $G$, where an OPS $\varphi$ corresponds to the $K$-point
$$\gamma_\varphi:\Spec K\to\Spec\CC[t,t^{-1}]=\bG_m\xrightarrow{\varphi}G.$$
Thus, we may always put $\ttilde{f}$ into a ``normal form," i.e. pick coset representatives of the $\hhat{\sigma}_{x_i}$ such that, for each $x\in S$, we have $g_x=\gamma_{\varphi_x}$ for some dominant OPS $\varphi_x$.

We also have the following identities between $\cD$ and $\cD_{\Bh}$, where $\cD_{\Bh}$ is the determinant of cohomology line bundle on $\Bh_G(\nu)$. The exact sequence
\begin{equation}\label{sesbh}0\to\cE\to\nu_*\cF\to\bigoplus_{x\in S}Q_x\to 0\end{equation}
for a Bhosle bundle $(\cF,\vec{q})$ shows that there is an isomorphism
\begin{equation}\label{naturaldbh}\pi^*\cD\xrightarrow\sim\cD_\Bh\otimes\bigotimes_{x\in S}\det\cQ_x,\end{equation}
where $\cQ_x$ is the ``universal $Q_x$" vector bundle. If $(\cE,\tau)$ is a $G$-bundle on $C_0$ (not just a singular $G$-bundle), then the exact sequence
\begin{equation}\label{sesvb}0\to\cE\to\nu_*\nu^*\cE\to\bigoplus_{x\in S}\cE|_x\to 0\end{equation}
gives a canonical isomorphism
\begin{equation}\label{dnue}\cD(\cE)\cong\cD(\nu^*\cE)\otimes\bigotimes_{x\in S}\det\cE|_x\cong\cD(\nu^*\cE),\end{equation}
where the second map is given by $\det\hhat{\tau}:\det\cE\xrightarrow\sim\cO$. If $(\cE,\tau)$ is induced by a Bhosle bundle $(\cF,\sigma,\vec{q})$, then we have a natural isomorphism of exact sequences (\ref{sesbh}) and (\ref{sesvb}) by lemma \ref{bhisos}, and the two identities (\ref{naturaldbh}) and (\ref{dnue}) get related by the following lemma.

\begin{lemma}\label{nonsense} Let $B$ be a ring. For an exact sequence of finitely generated projective $B$-modules
$$0\to M_1\to M_2\to\cdots\to M_n\to 0,$$
the canonical isomorphism $\det M_\bullet\xrightarrow\sim B$ is functorial with respect to isomorphisms of exact sequences.\end{lemma}
\begin{proof} We mean that, for any isomorphism $f_\bullet:M_\bullet\to N_\bullet$, the following diagram commutes:
\[\begin{tikzcd}
\det M_\bullet\arrow{dr}\arrow{r}{\det f_\bullet}&\det N_\bullet\arrow{d}\\
&B\end{tikzcd}\]
This is easy to check for a short exact sequence, and the general case can be done by induction.\end{proof}

Now we are ready for the pole calculation. In the following proposition, note that any free rank one $A$-module has a well-defined valuation function -- denoted ``$\ord$" -- given by picking an isomorphism to $A$ (the choice of isomorphism does not affect the valuation). Recall that $\kappa_{x_i}=\det q_{x_i}\circ(\det\hhat{\sigma}_{x_i})^{-1}$ for a descending $G$-bundle $(\cF,\sigma,\vec{q})$ and node $x\in S$. Let us also write $g_x=-\hhat{\sigma}_{x_2}q_{x_2}^{-1}q_{x_1}\hhat{\sigma}_{x_1}^{-1}$, and let $\varphi_x$ be the dominant coweights such that $\gamma_{\varphi_x}\in G(A)\bs G(K)/G(A)$ represents the double coset of $g_x$ (cf. the discussion at the beginning of the section).

\begin{proposition}\label{pole}Let $s$ be a section in the $\lambda$-component of $\mH^0(\Bun_G(C_0),\cD^l)$, where $\lambda$ is an assignment of a level $\leq ld_V$ dominant integral weight $\lambda_x$ to each node $x\in S$ (see lemma \ref{lambdadecomp}). Let $(\cF,\sigma,\vec{q})$ be a family of descending $G$-bundles given by a map $\ttilde{f}$ as in situation \ref{situation'}. Then $\ttilde{f}^*s$ vanishes to order at least
\begin{equation}\label{poleorder}\sum_{x\in S}(l\cdot\ord\kappa_{x_1}+w_0\lambda_x(\varphi_x))\end{equation}
at $t=0$, where $w_0$ is the longest element of the Weyl group of $G$.\end{proposition}

\begin{proof} Let $(\cE,\tau)$ be the family of singular $G$-bundles on $C_0$ induced by $(\cF,\sigma,\vec{q})$. As a section on $\Bun_G(C)$, we have (using lemma \ref{lambdadecomp})
$$s|_{(\cF,\sigma)}=\beta\otimes\bigotimes_{x\in S}v_x\otimes v_x^*$$
for some $\beta\in\cD(\cF)^l$, $v_x\in V^{\lambda_x}\otimes A,$ $v_x^*\in V^{\lambda_x^*}\otimes A$. We need to transfer $\beta$ back to $\cD(\cE_K)^l$ using isomorphism (\ref{dnue}), then apply isomorphism (\ref{naturaldbh}) to get an element of $D(\cF_K)^l\otimes\bigotimes_{x\in S}(\det Q_x)^l.$ As $\nu^*\cE_K$ and $\cF_K$ are isomorphic Bhosle bundles, there is an induced isomorphism of exact sequences (\ref{sesbh}) and (\ref{sesvb}). The isomorphism $\cE|_{x\times K}\to Q_x\otimes K$ is the composition
\begin{equation}\label{exq}\cE|_{x\times K}\xrightarrow\sim\nu^*\cE|_{x_1\times K}\xrightarrow\sim\cF|_{x_1\times K}\xrightarrow{q_{x_1}}Q_x\otimes K.\end{equation} 
Consider the diagram
\[
\begin{tikzcd}
\cD(\nu^*\cE_K) \arrow{r}{(\det\hhat{\tau})^{-1}}\arrow{d} & \cD(\nu^*\cE_K)\otimes\bigotimes_{x\in S}\det\cE|_{x\times K} \arrow{r}{\text{(\ref{sesvb})}} \arrow{d} & \cD(\cE_K)\arrow{d}{\text{(\ref{sesbh})}}\\
\cD(\cF_K)\arrow{r}{(\det\hhat{\sigma})^{-1}}&\cD(\cF_K)\otimes\bigotimes_{x\in S}\det\cF|_{x_1\times K}\arrow{r}{\det q_{x_1}}&\cD(\cF_K)\otimes\bigotimes_{x\in S}\det Q_x.
\end{tikzcd}
\]
The right-hand square is commutative by lemma \ref{nonsense} and equation (\ref{exq}), and the left-hand square is commutative because the isomorphism $\nu^*\cE_K\to\cF_K$ respects the singular $G$-bundle structures. Hence, the diagram commutes. The ``canonical route" (given by lemma \ref{lambdadecomp} and isomorphism (\ref{naturaldbh})) to transfer $s$ from $\Bun_G(C)$ to $\Bh_G(\nu)$ is to move the element $\beta\in\cD(\cF_K)^l\cong\cD(\nu^*\cE_K)^l$ along the top row and down the last column. This is the same as the map along the bottom row, which is just tensoring with $\bigotimes_{x\in S}\kappa_{x_1}$. Thus,
$$s|_{(\cF,\sigma,\vec{q})}=(\prod_{x\in S}v_x^*(\gamma_{\varphi_x}v_x))\beta\otimes\bigotimes_{x\in S}\kappa_{x_1}^l$$
as a rational section on $\Bh_G(\nu)$. Writing $v$ and $v^*$ as sums of weight vectors, we see that the order of $v^*(\gamma_\varphi v)$ at $t=0$ is at least $w_0\lambda(\varphi)$, because $w_0\lambda$ is the lowest weight of $V^\lambda$. This proves the proposition.\end{proof}

\subsection{Conclusion of proof of theorem \ref{sectionextensiontheorem}}

To conclude the proof of theorem \ref{sectionextensiontheorem}, we have to show that the quantity appearing in equation (\ref{poleorder}) is nonnegative for type A and C groups. This shows that the section $\ttilde{f}^*s$ from situation \ref{situation'} has no pole at $t=0$. For this we need a simple lemma.

\begin{lemma}\label{alpha_k} Let $M,N$ be free $A$-modules of the same rank, and $q:M\oplus M\to N$ a surjective $A$-module map such that $q_1,q_2$ are isomorphisms over $K$. Suppose there is an $A$-basis $\{e_i\}$ of $M$ with respect to which $q_2^{-1}q_1$ is a diagonal matrix $\diag(t^{a_1},\dots,t^{a_n})$. Then $\ord\det q_1$ is the sum of the $a_i$ which are nonnegative.\end{lemma}
\begin{proof}We may find a $K$-basis $\{f_i\}$ of $N\otimes_AK$ such that $q_1:e_i\mapsto f_i$, $q_2:e_i\mapsto t^{-a_i}f_i$. The elements $f_i$ must lie in the $A$-submodule $N\subseteq N\otimes_AK$. Let $m_i\geq 0$ be maximal such that $f_i'=t^{-m_i}f_i$ remains in $N$. Note $m_i\geq a_i$ because $q_2(e_i)=t^{-a_i}f_i\in N$, so $\mrm{image}(q)$ consists of $A$-linear combinations of the $f_i'$. Hence the $f_i'$ form a basis of $N$, and we may assume $M=N=A^r$, $q_1=\diag(t^{m_i})$, $q_2=\diag(t^{m_i-a_i})$. As $q$ is surjective, for each $i$ we either have $m_i=a_i$ or we have $m_i=0$ and $a_i< 0$.\end{proof}

\begin{proposition}\label{nopole} Suppose $G$ is a simple Lie group of type A or C, and the representation $V$ is chosen as follows:
\begin{enumerate}[label=(\roman*)]
\item if $G=\SL(n)$, choose $V=W\oplus W^*$, where $W=\CC^n$ is the standard representation;
\item if $G=\Sp(2n)$, choose $V=\CC^{2n}$ the standard representation.
\end{enumerate}
Then the quantity in equation (\ref{poleorder}) is nonnegative.\end{proposition}
\begin{proof} We will show that, for each node $x\in S$, we have
\begin{equation}\label{poleneed}l\cdot\ord\kappa_{x_1}+w_0\lambda_x(\varphi_x)\geq 0.\end{equation}
Thus, we will drop the subscripts indicating the node, and just consider an arbitrary dominant coweight $\varphi$ and a dominant weight $\lambda$ of level $\leq ld_V$. Assume that the representation $G\hookrightarrow\SL(V)$ sends the maximal torus of $G$ into the subgroup of diagonal matrices. Then by lemma \ref{alpha_k} the term $\ord\kappa$ in (\ref{poleneed}) is equal to the sum of all the nonnegative diagonal entries in the matrix representation of $\varphi\in\mf{g}\subset\mf{sl}(V)$.

The proof of the proposition is now to just check cases (i) and (ii). For simplicity we replace $w_0\lambda$ in equation (\ref{poleneed}) with $-\lambda$, which is fine since $\lambda$ and $-w_0\lambda$ have the same level.

\tbf{Case (i):} Give $V=W\oplus W^*$ the standard basis $e_1,\dots,e_n,e_1^*,\dots,e_n^*$. In this basis we have
$$\varphi=\diag(\varphi_1,\dots,\varphi_n,-\varphi_1,\dots,-\varphi_n),$$
for some integers $\varphi_1\geq\cdots\geq\varphi_n$ with $\varphi_1+\cdots+\varphi_n=0$. Let $\lambda=(\lambda_1\geq\cdots\geq\lambda_n=0)$ be a dominant weight of $\SL(n)$ at level $\lambda_1\leq ld_V=2l$. Then the inequality (\ref{poleneed}) becomes
\begin{equation}\label{slpole}l(|\varphi_1|+\cdots+|\varphi_n|)\geq\lambda_1\varphi_1+\cdots+\lambda_n\varphi_n.\end{equation}
Let $j$ be the index such that $\varphi_j\geq 0$ and $\varphi_{j+1}<0$. Since $\varphi_1+\cdots+\varphi_n=0$, we have
\[\begin{split}l\cdot(|\varphi_1|+\cdots+|\varphi_n|)&=2l\cdot (\varphi_1+\cdots+\varphi_j)\\
&\geq\lambda_1(\varphi_1+\cdots+\varphi_j)\\
&\geq\lambda_1\varphi_1+\cdots+\lambda_n\varphi_n.\end{split}\]
Thus inequality (\ref{slpole}) holds.

\tbf{Case (ii):} Our conventions will follow \cite{fultonharris}, lecture 16. A dominant coweight of $\mf{sp}(2n)$ is represented by a diagonal matrix $\varphi=\diag(\varphi_1,\dots,\varphi_n,-\varphi_1,\dots,-\varphi_n)$ with $\varphi_1\geq\cdots\geq\varphi_n\geq 0$ integers. A dominant weight is given by a decreasing sequence of integers $\lambda=(\lambda_1\geq\cdots\geq\lambda_n\geq0)$, and the level of such a weight is $(\lambda,\theta)=\lambda_1$. Since $V$ has Dynkin index $d_V=1$, we assume $\lambda_1\leq l$. Then inequality (\ref{poleneed}) is
$$l\cdot(\varphi_1+\cdots+\varphi_n)\geq\lambda_1\varphi_1+\cdots+\lambda_n\varphi_n,$$
which clearly holds.

\end{proof}

\begin{example}\label{polebd} Proposition \ref{nopole} fails when $V$ is the standard representation of a type B or D Lie group (corresponding to the Lie algebras $\mf{so}(2n+1)$, $\mf{so}(2n)$ repsectively). Note that for these groups the standard representation has Dynkin index $d_V=2$. A different choice of representation will not help, because, for any other representation, we will get sections of $\cD(V)^m$ for a large $m$ that still have poles (in fact we can take $m=1$, since every representation of a type B or D Lie group has Dynkin index divisible by 2).

Here is a counterexample to proposition \ref{nopole} for type $\text{B}_n$. Let $l=1$, and choose $\lambda=2\omega_1$ for one of the nodes and $\lambda=0$ for all the others. So $\lambda$ is a weight of level $2\leq ld_V$ and defines a component
\begin{equation}\label{2o1}\mH^0(\Bun_G(C),\cD(V)\otimes\scr{E}^{\lambda}_{x_1}\otimes\scr{E}^{\lambda^*}_{x_2})\subset\mH^0(\Bun_G(C_0),\cD(V)),\end{equation}
as in lemma \ref{lambdadecomp}. Let $\varphi=\diag(\varphi_1,\dots,\varphi_n,-\varphi_1,\dots,-\varphi_n,0)$ be a dominant coweight of $\mf{so}(2n+1)$, where $\varphi_1\geq\cdots\geq\varphi_n\geq 0$ are integers. Then the inequality (\ref{poleneed}) is
$$\varphi_1+\cdots+\varphi_n\geq 2\varphi_1,$$
which can clearly fail, e.g. if $\varphi=(1,0,\dots,0,-1,0,\dots,0,0)$. Thus sections in (\ref{2o1}) may not extend to the normalized stack of singular bundles.

One hope is that perhaps the $\lambda$-component (\ref{2o1}) vanishes. However, this is not the case. Suppose the normalization is $C=\bP^1$. Then
$\mH^0(\Bun_G(C),\cD(V)\otimes\scr{E}^{\lambda}_{x_1}\otimes\scr{E}^{\lambda^*}_{x_2})$
is identified with the dual of
$$\frac{V_\lambda\otimes V_\lambda^*}{\mf{g}\cdot(V_\lambda\otimes V_\lambda^*)+\im T^3,}$$
 where $T=f_\theta^{(2)}$ is the lowest root vector of $\mf{g}$ acting on the second tensor factor of $V_\lambda\otimes V_\lambda^*$ (see e.g. \cite{qhorn}, second proof of lemma 6.5). It is easy to verify directly that, for $\mf{g}=\mf{so}(5)$ and $\lambda=2\omega_1$, we have $T^3=0$ and 
$$\frac{V_\lambda\otimes V_\lambda^*}{\mf{g}\cdot(V_\lambda\otimes V_\lambda^*)}$$
is nonzero (in fact one-dimensional, spanned by any nonzero weight vector of weight zero).

However we can control the pole order in the following way. Recall that $\varphi$ is the ``type" of the family parametrized by $\Spec k[[t]]\to\Bh_G(\nu)$, i.e. we have $q_2^{-1}q_1=t^{\varphi}$ for this family of Bhosle bundles. Unlike $\lambda$, the $\varphi$ does not depend on $l$ in any way, so we have allowed it to be an arbitrary dominant coweight. However, since the boundary of $\Bun_G(C)$ in the Bhosle stack is the divisor $E$ where $\det q_1=0$, we should require that $\det q_1$ vanishes to order exactly one at $t=0$ in order to compute the correct pole order (although it makes no difference for detecting the existence of a pole). By lemma \ref{alpha_k}, there are only finitely many $\varphi$ with this property, and it follows that the pole order is bounded above by a number $N$ that scales linearly with $l$. Then we can replace $\cL=\cD_{\Bh}$ with $\cL'=\cL(NE)$, and, since sections over $\Bun_G(C)$ extend to $\cL'$, we should be able to carry out the rest of the paper in the same way to generalize our results to arbitrary groups. What remains to be shown is that the line bundle $\cL'$ descends to an ample line bundle on the appropriate moduli space.\end{example}

\section{Moduli of singular $G$-bundles}\label{part2}

In this section, we summarize the GIT construction of the moduli space of singular $G$-bundles due to Schmitt and Mu\~{n}oz-Casta\~{n}eda (\cite{schmc}), which we will use to prove theorem \ref{mainthm} in section \ref{proofmainthm}. We continue to focus on the case of a nodal curve, but all of the material in this section can be generalized to higher dimensional smooth varieties (see \cite{gomezlangerschmittsols}).

\subsection{Parameter spaces of singular $G$-bundles}\label{paramspaces}

Fix an ample line bundle $\cL$ on $C_0$, and let $Q$ be the quot scheme of coherent quotients $q:W\otimes\cL^{-n}\to\cE$ with Hilbert polynomial $\mP$, where $W$ is a vector space of rank $P(n)$. Let $\ttilde{\cE}$ be the universal quotient sheaf over $C_0\times Q$. To build a parameter space of singular $G$-bundles, we recall the following well-known result.

\begin{theorem}\label{hom}\emph{(\cite{fgaexplained}, theorem 5.8)} Let $p:Z\to S$ be a projective morphism to a Noetherian scheme $S$, and let $\cF,\cG$ be coherent $\cO_Z$-modules. If $\cG$ is flat over $S$, then the functor that sends an $S$-scheme $T$ to the set $\Hom_{Z_T}(\cF_T,\cG_T)$ is representable by a scheme $\mH_{Z/S}(\cF,\cG)$ which is affine and finite type over $S$.\end{theorem}

It is easy to show as a consequence:

\begin{corollary}\label{homalg} Let $Z\to S$ be a projective morphism to a Noetherian scheme $S$. If $\cS_\bullet$ is a quasicoherent, finitely generated, graded $\cO_Z$-algebra,  then the functor that sends a scheme $T$ to the set of $\cO_{Z_T}$-algebra homomorphisms $\cS_T\to\cO_T$ is representable by a finite type affine $S$-scheme.\end{corollary}
\begin{proof} Pick a generating submodule $\cF\subseteq\cS_\bullet$ that is coherent, and note that there is a closed subscheme of $\mH_{Z/S}(\cF,\cO)$ parametrizing morphisms $\cF\to\cO$ such that $\Sym^*\cF\to\cO$ factors through the multiplication map $\Sym^*\cF\to\cS_\bullet$.\end{proof}

Using this, we define two parameter schemes of singular $G$-bundles (an affine version and a projective version), as follows.

\begin{definition}\label{ttildeq} Define the affine parameter space of singular $G$-bundles $\ttilde{\mbf{Q}}$ to be the affine $Q$-scheme granted by corollary \ref{homalg} in the case $Z=C_0\times Q$, $S=Q$, $\cS=\Sym^*(V\otimes\ttilde{\cE})^G$. Then $\ttilde{\mbf{Q}}$ represents the functor that sends a scheme $T$ to the set of pairs $(q,\tau)$ consisting of a quotient $(q:W\otimes\cL_T^{-n}\to\cE)\in Q(T)$ and an algebra homomorphism
$$\tau:\Sym^*(V\otimes\cE)^G\to\cO_{C_0\times T}$$
on the quotient sheaf.\end{definition}

This scheme has a natural projectivization $\mbf{Q}$, which will be our projective parameter space of singular $G$-bundles.

\begin{proposition}\label{mbfq} There is a projective $Q$-scheme $\mbf{Q}$ parametrizing, over a scheme $T$, the triples $(q,\cM,\tau)$ consisting of a quotient $(q:W\otimes\cL_T^{-n}\to\cE)\in Q(T)$, a line bundle $\cM$ on $T$, and a morphism of graded $\cO_{C_0\times T}$-algebras
$$\tau:\Sym^*(V\otimes\cE)^G\to\cO_{C_0}\otimes\Sym^*\cM,$$
which is surjective in large degree. There is a surjective rational map $\ttilde{\mbf{Q}}\dashedrightarrow\mbf{Q}$ defined away from the zero section.\end{proposition}

\begin{proof}The universal quotient and the Reynolds operator induce a closed embedding $\ttilde{\mbf{Q}}\hookrightarrow Q\times\mA^{N+1}$, where
$$\mA^{N+1}=\bigoplus_{i=0}^d\Hom_\CC(\Sym^i(V\otimes W),\mH^0(\cL^{in}))$$
for sufficiently large $d>0$. Since $\pi:Q\times(\mA^{N+1}-0)\to Q\times\bP^N$ is a $\bG_m$-bundle and $\ttilde{\mbf{Q}}$ is $\bG_m$-stable, there exists by descent theory a unique closed subscheme $\mbf{Q}\subseteq Q\times\bP^N$ with $\pi^{-1}(\mbf{Q})=\ttilde{\mbf{Q}}-(Q\times 0).$ Since $\pi^{-1}(\mbf{Q})\to\mbf{Q}$ is a $\bG_m$-bundle, $\mbf{Q}$ represents the stack quotient $[(\ttilde{\mbf{Q}}-(Q\times 0))/\bG_m]$ and has the universal property that to give a map $T\to\mbf{Q}$ is to give a map $T\to Q$, a $\bG_{m,T}$-torsor $M\to T$, and a morphism of $Q$-schemes $M\to\ttilde{\mbf{Q}}-(Q\times 0)$. This is easily seen to be equivalent to the description of $\mbf{Q}$ given in the proposition.\end{proof}

Next we want a polarization of $\mbf{Q}$. We have already embedded $\mbf{Q}\hookrightarrow Q\times\bP^N$, and, for a large enough $m>0$, we may follow this with Grothendieck's embedding
$$Q\hookrightarrow\Gr(W\otimes\mH^0(\cL^m),f),$$
$$(q:W\otimes\cL^{-n}\twoheadrightarrow\cE)\mapsto(W\otimes\mH^0(\cL^m)\twoheadrightarrow\mH^0(\cE\otimes\cL^{m+n}))$$ to get a $\GL(W)$-equivariant embedding
\begin{equation}\label{embQ'}\mbf{Q}\hookrightarrow\Gr(W\otimes\mH^0(\cL^m),f)\times\bP^N,\end{equation}
where $\Gr(\cdots)$ is the Grassmannian of quotients. As above, we also have the embedding
$$\ttilde{\mbf{Q}}\hookrightarrow\Gr(W\otimes\mH^0(\cL^m),f)\times\mA^{N+1},$$
that makes $\ttilde{\mbf{Q}}$ the (partial) affine cone over $\mbf{Q}$ with respect to the embedding (\ref{embQ'}).

\begin{definition}\label{taxa} Let $L_m(k_1,k_2)$ be the pullback of the $\GL(W)$-linearized line bundle $\cO(k_1)\boxtimes\cO(k_2)$ under the embedding (\ref{embQ'}).\end{definition}

\subsection{Semistable tensor fields}

Semistability for singular $G$-bundles is defined in terms of their associated ``tensor fields." The rough idea is to pick a generating submodule of $\Sym^*(V\otimes\cE)^G$, use the Reynolds operator to drop the $G$-invariance requirement, and then ``homogenize" (we will give the explicit construction in the next subsection). The result is a very simple object of the following form, which retains all of the information about the singular $G$-bundle (up to scalars) and allows us to get a much simpler definition of semistability.

\begin{definition}A tensor field on a sheaf $\cF$ is a nonzero morphism $\varphi:(\cF^{\otimes b})^{\oplus c}\to\cO_{C_0}$ for some $b,c$.\end{definition}

Gomez-Sols introduced the following definition of semistability for tensor fields. Recall that a weighted filtration of a sheaf $\cF$ is a pair $(\cF_\bullet,l_\bullet)$ consisting of an increasing filtration
$$\cF_\bullet=(0\subset\cF_1\subset\cdots\subset\cF_{p+1}=\cF)$$
by distinct subsheaves and a sequence of positive rational numbers
$$l_\bullet=(l_1,\dots,l_p).$$
Let $a=\deg\cL$. Given a tensor field $(\cF,\varphi)$ of $\rk\cF=r$ and a weighted filtration $(\cF_\bullet,l_\bullet)$, define the vector
$$\lambda(l_\bullet)=\sum_{i=1}^pl_i\xi_{r_i},$$
where $r_i=\rk\cF_i$ and
$$\xi_j=(aj-ar,\dots,aj-ar,aj,\dots,aj)$$
with $aj-ar$ repeated $aj$ times and $aj$ repeated $ar-aj$ times. Write $\lambda(l_\bullet)=(\lambda_1,\dots,\lambda_{ar})$ and define
$$\mu(\cF_\bullet,l_\bullet,\varphi)=-\min\{\lambda_{ar_{i_1}}+\cdots+\lambda_{ar_{i_b}}:\varphi|_{(\cF_{i_1}\otimes\cdots\otimes\cF_{i_b})^{\oplus c}}\neq 0\},$$
where the min is taken over all $i_1,\dots,i_b\in\{1,\dots,p+1\}$, possibly nondistinct. (All the ranks are scaled by $a$ since they might not be integers).

\begin{definition} (\cite{gomez-sols}, definition 1.3) Let $\delta\in\QQ_{>0}$. A tensor field $(\cF,\varphi)$ is $\delta$-semistable if $\cF$ is torsion-free and
\begin{equation}\label{deltass}\sum_{i=1}^pl_i(\rk\cF_i\cdot\mP_\cF-\rk\cF\cdot\mP_{\cF_i})+\delta\mu(\cF_\bullet,l_\bullet,\varphi)\geq 0\end{equation}
for every weighted filtration $(\cF_\bullet,l_\bullet)$. We say $(\cF,\varphi)$ is $\delta$-stable if strict inequality holds for every such $(\cF_\bullet,l_\bullet)$.\end{definition}

\subsection{Tensor field associated to a singular $G$-bundle}\label{associatedtensorfield}

Recall that $\ttilde{\cE}$ is the universal quotient sheaf on the quot scheme $Q$. Picking a $d$ such that $\Sym^*(V\otimes\ttilde{\cE})^G$ is generated in degree $\leq d$, we may associate to every point $(\cE,\tau)\in\mbf{Q}$ a tensor field
\begin{equation}\label{tfmap}\bigoplus_{\vec{a}}(V\otimes\cE)^{\otimes d!}\to\bigoplus_{\vec{a}}\bigotimes_{i=0}^d\Sym^{a_i}\Sym^i(V\otimes\cE)\xrightarrow{R_G}\bigoplus_{\vec{a}}\bigotimes_{i=0}^d\Sym^{a_i}\Sym^i(V\otimes\cE)^G\xrightarrow{\tau}\cO_{C_0},\end{equation}
where $R_G$ is the Reynolds operator and the sum is over $\vec{a}=(a_0,\dots,a_d)$ with $a_1+2a_2+\cdots+da_d=d!$. It is shown in \cite{mcthesis}, sect. 2.2.3 (or \cite{gomezlangerschmittsols}, sect. 5) that this assignment is injective on $\bG_m$-equivalence classes and defines a proper injective morphism from $\mbf{Q}$ into a parameter space of tensor fields. We therefore define:

\begin{definition} A point of $\mbf{Q}$ is $\delta$-semistable if the associated tensor field is $\delta$-semistable.\end{definition}

\subsection{Characterization of GIT semistability}\label{sectionss}

Let $Q(1)$ be the closure of the set of torsion-free, uniform rank $r$ quotient sheaves in $Q$, and let $\mbf{Q}(1)=\mbf{Q}\times_QQ(1)$. Recall the $\SL(W)$-linearized ample line bundle $L=L_m(k_1,k_2)$ on $\mbf{Q}$. Following Simpson's approach for semistable sheaves (\cite{simpson}), Gomez-Sols proved that $L$-semistability and $\delta$-semistability coincide in the following sense.

\begin{theorem}\label{ss}\emph{(\cite{gomez-sols}, theorem 3.6)} Assume $m,n$ are sufficiently large. There is a number $\alpha$ such that, if $\frac{k_2}{k_1}=\alpha$, then a point $(q,\tau)\in\mbf{Q}(1)$ is $L$-semistable if and only if $(\cE,\tau)$ is a $\delta$-semistable singular $G$-bundle and $W\to H^0(\cE\otimes\cL^n)$ is an isomorphism.\end{theorem}

\subsection{Semistability for large values of $\delta$}

By theorem \ref{ss}, we thus get a projective moduli space of singular $G$-bundles $\mbf{Q}(1)\sslash_L\SL(W)$. The downside is that the definition of $\delta$-semistability is somewhat difficult, and the $\delta$-semistable singular $G$-bundles are too large a class of objects, in that we may not get a $G$-bundle over $C_0-S$ or even over a dense open subset of $C_0$. However, Schmitt and Mu\~{n}oz-Casta\~{n}eda have recently shown that, for large values of $\delta$, the moduli space parametrizes only \emph{honest} singular $G$-bundles which are semistable in the sense of definition \ref{honestss}. There are analagous results for smooth varieties of arbitrary dimension, even in positive characteristic (\cite{gomezlangerschmittsols}, theorems 5.4.1 and 5.4.4).

\begin{theorem}\label{sch-mc}\emph{(\cite{schmc}, theorem 3.3)} There is a $\delta_0>0$ such that, for $\delta>\delta_0$, the following hold:\begin{enumerate}
\item any $\delta$-semistable singular $G$-bundle with Hilbert polynomial $P$ is honest;
\item for any honest singular $G$-bundle with Hilbert polynomial $P$, $\delta$-(semi)stability is equivalent to (semi)stability as in definition \ref{honestss}.\end{enumerate}\end{theorem}

\subsection{Set-up for finite generation}\label{setupforfg}

We can now define the moduli space which will be used to prove theorem \ref{mainthm} in the next section. We introduce the following schemes related to $\mbf{Q}$:

\begin{itemize}
\item $\mbf{Q}^0\subseteq\mbf{Q}(1)$ the open subset parametrizing torsion-free, honest singular $G$-bundles such that the map $W\to\mH^0(\cE\otimes\cL^n)$ is an isomorphism (in particular $h^1(\cE\otimes\cL^n)=0$);
\item $\mbf{Q}^G\subseteq\mbf{Q}^0$ the open subscheme of $G$-bundles;
\item $\mbf{M}$ the normalization of $\bbar{\mbf{Q}^G}$ (closure taken in $\mbf{Q}$);
\item $\mbf{M}^0\subseteq\mbf{M}$ the preimage of $\mbf{Q}^0$.
\end{itemize}

\noindent Define $\cX=\cX_m(k_1,k_2)$ as the GIT quotient 
$$\mbf{M}\sslash_{L_m(k_1,k_2)}\SL(W)$$ 
with respect to the following choices:
\begin{itemize}
\item $\cL$ is an ample line bundle on $C_0$;
\item $\mP=\mP_{\cO^{\oplus r}}$;
\item $\delta>\delta_0$ as in theorem \ref{sch-mc} and $m,n,k_1,k_2$ are chosen as in theorem \ref{ss}.
\end{itemize}

\noindent Remember, as mentioned in the introduction, that we are not sure if $\bbar{\mbf{Q}^G}$ contains all of the honest singular $G$-bundles! But, it is necessary for us to work with this smaller moduli space, e.g. the key result of section \ref{surjection} -- proposition \ref{lift} -- only applies to singular $G$-bundles which are in the closure of $\Bun_G(C_0)$.

\section{Proof of theorem \ref{mainthm}}\label{proofmainthm}

In the previous section, we constructed a polarized moduli space $(\cX,L)$, where $L=L_m(k_1,k_2)$ (or a sufficiently large multiple thereof that descends to $\cX$). After identifying $L$ with determinant of cohomology, we will establish an injection
$$\mH^0(\cX,L)\hookrightarrow\mH^0(\Bun_G(C_0),\cD(V)^l)$$
for sufficiently divislbe $l$, and show that this map is an isomorphism by way of theorem \ref{sectionextensiontheorem}. This will prove theorem \ref{mainthm}.

\subsection{Line bundle identities}\label{linebundleidentities}

Let $\bP^N$ be the projective space from section \ref{paramspaces} such that $\mbf{Q}\hookrightarrow Q\times\bP^N$, and recall that $L_m(k_1,k_2)$ is a tensor product of line bundles
$$L_m(k_1,k_2)=\cO_Q(k_1)\boxtimes\cO_{\bP^N}(k_2).$$
For a vector bundle $\cG$ on $C_0$, let $\cD_\cG$ be the line bundle on $Q$ whose fiber over a quotient sheaf $\cE$ is the determinant of cohomology of $\cE\otimes\cG$, i.e.
$$\cD_\cG|_\cE=\det\mH^0(\cE\otimes\cG)^*\otimes\det\mH^1(\cE\otimes\cG).$$ 
The standard determinant of cohomology line bundle (corresponding to $\cG=\cO_{C_0}$) will continue to be denoted $\cD$.

\begin{lemma}\label{lblemma}The line bundles $\cD_\cG$ have the following properties.\begin{enumerate}[label=(\roman*)]
\item $\cD_\cG$ has a natural $\GL(W)$-linearization, where a scalar matrix $t\in\GL(W)$ acts on $\cD_\cG$ as multiplication by $t^{-\chi(\cE\otimes\cG)}$.
\item $\cO_Q(-1)=\cD_{\cL^{m+n}}$.
\item If $p\in C_0$ is a smooth point, and $\cE$ is a torsion-free sheaf on $C_0$, then there is a natural isomorphism
$$\cD(\cE\otimes\cO(p))\cong\cD(\cE)\otimes\det(\cE|_p)^*,$$
which induces a $\GL(W)$-equivariant isomorphism of line bundles
$$\cD_{\cO(p)}\cong\cD\otimes\det(\ttilde{\cE}|_p)^*,$$
where $\ttilde{\cE}$ is the universal quotient sheaf on $Q\times C_0$.
\end{enumerate}\end{lemma}
\begin{proof} Items (i), (ii) are well known, see e.g. \cite{huylehn}. For (iii), since $\cE$ is a vector bundle in a neighborhood of $p$, there is an exact sequence
$$0\to\cE\to\cE\otimes\cO(p)\to\cE|_p\to0.$$
Then use the fact that determinant of cohomology is multiplicative with respect to short exact sequences.\end{proof}

\begin{proposition}\label{lbidentities}Assume $n$ is divisible by $g-1$, so that $\chi(\cE)$ divides $P(n)$, and let
$$l=-\frac{P(m+n)P(n)}{\chi(\cE)}$$
(note that $\chi(\cE)=r(1-g)$ is negative, so $l$ is positive). Then, over $\tbf{Q}^0$, there is a $\GL(W)$-equivariant isomorphism
$$\cO_Q(P(n))\cong\cD^l.$$\end{proposition}
\begin{proof}First note, since the map $W\to\mH^0(\cE\otimes\cL^n)$ is an isomorphism for $\cE\in Q^0$, the line bundle $\cD_{\cL^n}$ is trivial over $Q^0$, isomorphic to $\det W^*\otimes\cO_Q$. Furthermore, over $\tbf{Q}^0\times(C_0-S)$, the universal quotient sheaf $\ttilde{\cE}$ is a vector bundle with trivial determinant, so for any smooth point $p\in C_0$ the line bundle $\det(\ttilde{\cE}|_p)$ is trivial over $\tbf{Q}^0$. These two observations combined with (iii) of lemma \ref{lblemma} show that $\cD$ itself is trivial over $\tbf{Q}^0$, as is $\cO_Q(1)=\cD_{\cL^{m+n}}$.

Hence, in order to show $\cO_Q(P(n))\cong\cD^l$ as equivariant line bundles, we just have to show that they have the same linearization. The difference in linearization is given by a character of $\GL(W)$, which is determined by the action of the center of $\GL(W)$. Since a scalar matrix $t\in\GL(W)$ acts on $\cO_Q(P(n))=\cD_{\cL^{m+n}}^{-P(n)}$ by $t^{P(m+n)P(n)},$ and acts on $\cD^l$ by $t^{-l\chi(\cE)}$, our choice of $l$ ensures that $\cO_Q(P(n))\cong\cD^l$.
\end{proof}

\subsection{The injection $\mH^0(\mbf{M},L)^{\SL(W)}\hookrightarrow\mH^0(\Bun_G(C_0),\cD^l)$}\label{injsections}

Recall that $\cD$ also denotes the determinant of cohomology line bundle on $\Bun_G(C_0)$ with respect to the contraction map
$$\Bun_G(C_0)\to\Coh(C_0),$$
$$E\mapsto E\times^GV.$$
\noindent In this section we prove:

\begin{proposition}\label{injprop}If $n,k_1,k_2$ are sufficiently divisible, then there is an injection
$$\mH^0(\mbf{M},L_m(k_1,k_2))^{\SL(W)}\hookrightarrow\mH^0(\Bun_G(C_0),\cD^l),$$
where $l$ is given by
\begin{equation}\label{l}l=-\frac{k_1P(m+n)}{\chi(\cE)}.\end{equation}\end{proposition}

The proof of the proposition is as follows. In lemma \ref{qgsm} below, we show $\mbf{Q}^G$ is smooth, so that it is an open subscheme of $\mbf{M}$, hence restriction gives a map
$$\mH^0(\mbf{M},L)^{\SL(W)}\hookrightarrow\mH^0(\mbf{Q}^G,L)^{\SL(W)},$$
where $\ttilde{\mbf{Q}}^G$ is the preimage of $\mbf{Q}^G$ in $\ttilde{\mbf{Q}}$. Since $\mbf{Q}^G=\ttilde{\mbf{Q}}^G/\bG_m$ and the center of $\SL(W)$ acts trivially on sections of $L$ under the assumptions of the proposition, we have
$$\mH^0(\mbf{Q}^G,L)^{\SL(W)}=\mH^0(\mbf{Q}^G,L)^{\PGL(W)}=\mH^0(\ttilde{\mbf{Q}}^G,L)^{\GL(W)}.$$
The pullback of $\cO_{\bP^N}(1)$ to $\ttilde{\mbf{Q}}^G$ is trivial, so by proposition \ref{lbidentities}
$$\mH^0(\ttilde{\mbf{Q}}^G,L)^{\GL(W)}=\mH^0(\ttilde{\mbf{Q}}^G,\cD^l)^{\GL(W)}.$$
Note that $\ttilde{\mbf{Q}}^G$ is a $\GL(W)$-bundle over an open substack 
$$Y_{\cL^n}=[\ttilde{\mbf{Q}}^G/\GL(W)]\subset\Bun_G(C_0)$$
parametrizing $G$-bundles $E$ such that the associated vector bundle $\cE=E\times^GV$ has $\cE\otimes\cL^n$ globally generated and $h^1(\cE\otimes\cL^n)=0$. If $n$ is sufficiently large, then $Y_{\cL^n}$ contains the locus of semistable $G$-bundles, so its complement has codimension $\geq 2$ (see \cite{laszlo-sorger}, proof of 1.6, or \cite{laumon-rapoport}, section 3), thus
$$\mH^0(\ttilde{\mbf{Q}}^G,\cD^l)^{\GL(W)}=\mH^0(Y_{\cL^n},\cD^l)=\mH^0(\Bun_G(C_0),\cD^l).$$
It remains to show $\mbf{Q}^G$ is smooth.

\begin{lemma}\label{qgsm} $\mbf{Q}^G$ is a smooth, irreducible variety.\end{lemma}
\begin{proof} Since $\ttilde{\mbf{Q}}^G\to\mbf{Q}^G$ is smooth, it suffices to prove the lemma for $\ttilde{\mbf{Q}}^G$. But $\ttilde{\mbf{Q}}^G$ is a $\GL(W)$-bundle over the open substack $Y_{\cL^n}\subset\Bun_G(C_0)$ (\cite{wang}, section 4), and $\Bun_G(C_0)$ is smooth over $\Spec\CC$ and connected (\cite{belfakh}, proposition 5.1), so $\ttilde{\mbf{Q}}^G$ is smooth and irreducible.\end{proof}

\subsection{Conclusion of finite generation}

Using theorem \ref{sectionextensiontheorem}, we will show that the injection we have set up is an isomorphism and conclude finite generation of $\bigoplus_{k\geq 0}\mH^0(\Bun_G(C_0),\cD^k)$.

\begin{theorem}\label{isodkl} Let $G$ be a type A or C simple, simply connected Lie group, with the representation $V$ chosen as in proposition \ref{nopole}. Let $L=L_m(k_1,k_2)$, and assume $l,n,k_1,k_2$ are as in proposition \ref{injprop}. Then we have an algebra isomorphism
$$\bigoplus_{k\geq 0}\mH^0(\mbf{M},L^k)^{\SL(W)}\xrightarrow\sim \bigoplus_{k\geq 0}\mH^0(\Bun_G(C_0),\cD^{kl}).$$\end{theorem}
\begin{proof} As $\mbf{M}\to\bbar{\mbf{Q}}^G$ is finite, the $L$-semistable locus in $\mbf{M}$ is the preimage of that in $\mbf{Q}$, in particular it is contained in $\mbf{M}^0$ by theorem \ref{sch-mc}. Hence,
$$\mH^0(\mbf{M},L^k)^{\SL(W)}=\mH^0(\mbf{M}^0,L^k)^{\SL(W)}$$
by \cite{theta1}, lemma 4.15. Since sections of $\cD^{kl}$ on $\Bun_G(C_0)$ are identified with sections of $L^k$ on $\mbf{Q}^G$, and since $\mbf{M}$ is normal, we need to show that sections of $L^k$ over $\mbf{Q}^G$ have no pole at $t=0$ for a map $\Spec k[[t]]\to\mbf{Q}^0\cap\bbar{\mbf{Q}^G}$ sending the generic point into $\mbf{Q}^G$. As $k[[t]]$ is local, any such map factors through $\ttilde{\mbf{Q}}$, and the pullback of $L$ to $\ttilde{\mbf{Q}}^0$ is $\cD^l$ by proposition \ref{lbidentities} (where $\ttilde{\mbf{Q}}^0\subseteq\ttilde{\mbf{Q}}$ is the preimage of $\mbf{Q}^0$),  so we reduce to showing that sections of $\cD^l$ over $\Bun_G(C_0)$ have no pole at $t=0$ for a map $\Spec k[[t]]\to\SB_G^0(C_0)$ sending the generic point into $\Bun_G(C_0)$. This is what is shown in theorem \ref{sectionextensiontheorem}.\end{proof}

\begin{corollary}\label{finitegeneration} Let $G$ be a type A or C simple, simply connected Lie group. For any stable curve $C_0$ of genus $g\geq 2$, the algebra $\bigoplus_{k\geq 0}\mH^0(\Bun_G(C_0),\cD^k)$ is finitely generated.\end{corollary}
\begin{proof} Since $\bigoplus_{k\geq 0}\mH^0(\Bun_G(C_0),\cD^{kl})\cong\bigoplus_{k\geq 0}\mH^0(\mbf{M},L^k)$ is finitely generated, this follows from \cite{belgib}, lemma 8.4.\end{proof}

\section{Compactifications over $\bbar{\cM}_g$ and conformal blocks}\label{conformalblocks}

\subsection{Conformal blocks vector bundles}

Let $\bbar{\cM}_{g,n}$ be the stack of stable $n$-pointed curves of genus $g$, and let $\mf{g}$ be the Lie algebra of the simple, simply connected group $G$. For a positive integer $l$ and dominant integral weights $\vec{\lambda}=(\lambda^1,\dots,\lambda^n)$ of level $\leq l$, there is a vector bundle $\mbb{V}_{\mf{g},l,\vec{\lambda}}$ on $\bbar{\cM}_{g,n}$ called the vector bundle of conformal blocks.

Let us briefly recall its construction from \cite{fakh12}. Let $\hhat{\mf{g}}=(\mf{g}\otimes k((t)))\oplus k\cdot c$ be the affine Lie algebra of $\mf{g}$, where $c$ is central and the Lie bracket is defined by
$$[x\otimes f,y\otimes g]=[x,y]\otimes fg+(x,y)\mrm{Res}(f'g)c.$$
For each dominant integral weight $\lambda$ of $\mf{g}$ there is an irreducible representation $\cH_\lambda$ of $\hhat{g}$, and we put $\cH_{\vec{\lambda}}=\bigotimes_{i=1}^n\cH_{\lambda^i}$, which is an irreducible representation of
$$\hhat{\mf{g}}_n=(\mf{g}\otimes k((t))^{\oplus n})\oplus k\cdot c.$$ 
Suppose $\pi:\scr{C}\to S$ is a proper flat family of genus $g$ nodal curves parametrized by a smooth $k$-variety $S$, and let $\vec{p}=(p_1,\dots,p_n)$ be disjoint sections $p_i:S\to \scr{C}$ of $\pi$ whose images lie in the smooth locus of $\pi$. Assume that $S=\Spec A$ is affine, that $\scr{C}-(\bigcup_{i=1}^np_i(S))$ is affine with coordinate ring $B$, and that we are given isomorphisms $\eta_i:\hhat{\cO}_{\scr{C},p_i(S)}\xrightarrow\sim A[[t]].$ Then the $\eta_i$ make $\mf{g}\otimes_kB$ a Lie subalgebra of $\hhat{\mf{g}}_n\otimes_kA$, and we define
$$\mbb{V}_{\mf{g},l,\vec{\lambda}}(\scr{C},\vec{p})=\cH_{\vec{\lambda}}\otimes_kA/(\mf{g}\otimes_kB)\cdot(\cH_{\vec{\lambda}}\otimes_kA).$$
The case of a general $S$ can be dealt with by a descent argument (\cite{fakh12}, proposition 2.1, and the discussion following it).

For each conformal blocks bundle $\mbb{V}_{\mf{g},l,\vec{\lambda}}$, the sum $\bigoplus_{m\geq 0}\mbb{V}_{\mf{g},ml,m\vec{\lambda}}^\vee$ has a natural structure of a flat sheaf of algebras on $\bbar{\cM}_{g,n}$ (\cite{manon}). For a marked curve $(C_0,\vec{p})$, let $\Parbun_G(C_0,\vec{p})$ be the stack parametrizing $(E,s_1,\dots,s_n)$ consisting of a $G$-bundle $E\to C_0$ and points $s_i\in E_{p_i}/B$ for a fixed Borel subgroup $B\subset G$. Let $p:\Parbun_{G,g,n}\to\bbar{\cM}_{g,n}$ be the corresponding relative stack. In \cite{belfakh}, it is shown that for each conformal block $\mbb{V}_{\mf{g},l,\vec{\lambda}}$, there is a line bundle $\cL_{G,l,\vec{\lambda}}$ on $\Parbun_{G,g,n}$ with $p_*\cL_{G,l,\vec{\lambda}}\cong\mbb{V}_{\mf{g},l,\vec{\lambda}}^\vee$ (such isomorphisms exist for any family of stable curves, not just globally over $\bbar{\cM}_{g,n}$), and by \cite{belgib} theorem 9.2 this induces an algebra isomorphism
$$\bigoplus_{m\geq 0}\mH^0(\Parbun_G(C_0,\vec{p}),\cL_{G,l,\vec{\lambda}}(C_0,\vec{p})^m)\cong\bigoplus_{m\geq 0}\mbb{V}_{\mf{g},ml,m\vec{\lambda}}^\vee(C_0,\vec{p})$$
for all $(C_0,\vec{p})\in\bbar{\cM}_{g,n}$. The line bundle $\cL_{G,l,\vec{\lambda}}$ is related to determinant of cohomology in the following way. For a representation $V$ of $G$, let $\cN_{V,l,\vec{\lambda}}$ be the line bundle on $\Parbun_{G,g,n}$ whose fiber over a parabolic bundle $(E,s_1,\dots,s_n)\in\Parbun_G(C_0,\vec{p})$ is the tensor product of
\begin{itemize}
\item $[\det H^*(E\times^GV)\otimes\det H^*(V\otimes\cO_{C_0})^{-1}]^l,$
\item the fibers of the line bundles $E_{p_i}\times^B\CC_{-\lambda^i}\to E_{p_i}/B$ over the elements $s_i$.
\end{itemize}
If $V$ is irreducible, then $\cN_{V,l,\vec{\lambda}}\cong\cL_{G,d_Vl,\vec{\lambda}}$, where $d_V$ is the Dynkin index of $V$ \cite{belfakh}. In the case of no marked points ($n=0$), then even for reducible $V$ we have $\cN_{V,l}\cong\cL_{G,dl}$, where $d$ is the sum of the Dynkin indices of the irreducible summands of $V$, and thus
$$\bigoplus_{m\geq 0}\mH^0(\Bun_G(C_0),\cD(V)^{m})\cong\bigoplus_{m\geq 0}\mbb{V}_{\mf{g},md}^\vee(C_0)$$
for any $V$. Therefore, when $\mf{g}$ is $\mf{sl}(n)$ or $\mf{sp}(2n)$, corollary \ref{finitegeneration} and \cite{belgib} lemma 8.4 give finite generation of the conformal blocks algebra for any stable curve $C_0\in\bbar{\cM}_g$.

\subsection{Finite generation of the sheaf of conformal blocks algebras}

We would like to show that $\cA=\bigoplus_{m\geq 0}\mbb{V}_{\mf{g},m}^\vee$ is finitely generated as a sheaf of algebras on $\bbar{\cM}_g$. Recall that $\bbar{\cM}_g$ has a smooth atlas $H_g\to\bbar{\cM}_g$, where $H_g$ is a smooth, irreducible variety \cite{dm}. To show that $\cA$ is finitely generated, it suffices to show that there is a uniform constant $d$ such that, for each closed point $C_0\in\bbar{\cM}_g$, the fiber $\cA(C_0)$ is generated in degree $\leq d$ (for then $\cA$ is generated in degree $\leq d$ by Nakayama's lemma). We will prove this by showing that, for any family of stable curves $T\to\bbar{\cM}_g$ parametrized by a variety $T$, there is a constant $d=d(T)$ such that $\cA|_t$ is generated in degree $\leq d$ for all $t\in T$. The proof is by induction on $\dim T$, so that we are free to replace $T$ by a dense open subset.

Let me sketch the adjustments needed for the relative setting. Let $\pi:\scr{C}\to T$ be a family of stable curves parametrized by a variety $T$ (we can assume $T$ is smooth and irreducible), and fix a relatively ample line bundle $\cL$ on $\scr{C}/T$. All of the parameter schemes from section \ref{part2} have relative versions:
\begin{itemize}
\item $\ttilde{\mbf{Q}}_{\scr{C}/T}$ and $\mbf{Q}_{\scr{C}/T}$ are defined in the same way, but over a relative quot scheme 
$$Q_{\scr{C}/T}=\Quot_{\scr{C}/T}(W\otimes\cL^{-n},P);$$
\item $\ttilde{\mbf{Q}}_{\scr{C}/T}$ and $\mbf{Q}_{\scr{C}/T}$ again come with closed embeddings into $Q_{\scr{C}/T}\times_T\mA^{N+1}_T$ and $Q_{\scr{C}/T}\times_T\bP^N_T$;
\item there is a Grothendieck embedding $\mbf{Q}_{\scr{C}/T}\hookrightarrow\Gr\times_T\bP^N_T$, where 
$$\Gr=\Gr_T(\pi_*(W\otimes\cL^m),f)$$
is a relative Grassmannian, and we define $L=L_m(k_1,k_2)$ to be the relatively ample line bundle $\cO_{\Gr}(k_1)\boxtimes\cO_{\bP^N_T}(k_2)$ on $\mbf{Q}_{\scr{C}/T}$.
\end{itemize}

Define $\delta$-semistability in the relative case as fiberwise $\delta$-semistability over $T$. The number $\delta_0$ from theorem \ref{sch-mc} can be chosen uniformly over $T$ by \cite{mcthesis}, theorems 4.4.17 and 4.4.18 ($\delta_0$ depends ``only on numerical inputs," the only one of which that depends on the base curve is the degree of the polarization, but we can just use the canonical polarization). By \cite{mcuniformity}, theorem 2.16, the numbers $m,n$ from theorem \ref{ss} can be chosen uniformly over $T$ (the resulting $\alpha$ depends only on $P,m,n,\delta$). Since $L$-semistability coincides with $L_t$-semistability on the fibers $\mbf{Q}_t$, $t\in T$ (\cite{simpson}, lemma 1.13), the analogue of theorem \ref{ss} holds in the relative setting.

The schemes $\mbf{Q}^0_{\scr{C}/T},\mbf{Q}^G_{\scr{C}/T},\mbf{M}_{\scr{C}/T}$ can all be defined in the analagous way -- the only hitch is that the restriction of $\mbf{M}_{\scr{C}/T}$ to a point $t\in T$ is not the same as $\mbf{M}_{\scr{C}_t}$ since normalization does not commute with base change. But we are free to replace $T$ by an open subset, so we only need this to work generically.

\begin{lemma}\emph{(\cite{belgib}, lemma 10.3)} Let $U\subseteq Y$ be schemes over a smooth variety $S$ over a field of characteristic zero. Assume $U\to S$ is smooth and the fibers over closed points are irreducible. Also assume $Y\to S$ is proper. Let $N$ be the normalization of the closure of $U$ in $Y$. Then there is a dense open subset $V\subseteq S$ such that the fibers of $N$ over the closed points of $V$ are normal irreducible varieties.\end{lemma}

\noindent Thus, shrinking $T$, we can assume $\mbf{M}_{\scr{C}/T}|_t=\mbf{M}_{\scr{C}_t}$ for all $t\in T$. Let $\cB=(\bigoplus_{k\geq 0}q_*L^k)^{\SL(W)}$, where $q:\mbf{M}_{\scr{C}/T}\to T$ and $L=L_m(k_1,k_2)$ is the ample linearization on $\mbf{M}_{\scr{C}/T}$ (we may have to twist $L$ by a line bundle on $T$ in order for it to be ample, or we can just assume $T$ is affine). Replacing $L$ by a sufficiently large multiple, we have $\cB|_t=\bigoplus_{k\geq 0}\mH^0(\mbf{M}_{\scr{C}_t}, L_t^k)^{\SL(W)}$ for all $t\in T$. One may establish an injection of $\cB$ into a Veronese subalgebra $\cA^{(N)}=\bigoplus_{k\geq 0}\cA_{kN}$ as in section \ref{injsections}, and we know $\cB|_t\to\cA^{(N)}|_t$ is an isomorphism for all $t\in T$ by theorem \ref{isodkl}. Thus $\cB\cong\cA^{(N)}$. Since $\cB$ is a finitely generated $\cO_T$-algebra, so is $\cA$ by \cite{belgib} lemma 8.4.

\subsection{The family $\mf{X}\to\bbar{\cM}_g$ and modular interpretations}

Let $\cA$ be the conformal blocks algebra on $\bbar{\cM}_g$, and $\cX(C_0)=\mbf{M}\sslash_L\SL(W)$ the moduli space defined in subsection \ref{setupforfg}. By theorem \ref{isodkl},
$$\Proj \cA(C_0)\cong\cX(C_0)$$ 
for every closed point $C_0\in\bbar{\cM}_g$. The last section showed that $\cA$ is finitely generated, so taking $\mf{X}=\Proj\cA$ we have the following. Note that item 1. is by a theorem of Kumar-Narasimhan-Ramanathan, \cite{knr}.

\begin{theorem} Let $G$ be a simple Lie group of type A or C, and let $g\geq 2$. Then there is a flat, relatively projective family $\mf{X}\to\bbar{\cM}_g$ such that
\begin{enumerate}
\item the fiber over a smooth curve is Ramanathan's moduli space of semistable $G$-bundles;
\item the fiber over an arbitrary curve is a normalized moduli space of semistable honest singular $G$-bundles.
\end{enumerate}\end{theorem}

\noindent Note that the theorem gives an interpretation of the fibers of $\Proj\cA\to\bbar{\cM}_g$ over fixed stable curves, but we have not given an interpretation over families of stable curves, since the varieties $\cX$ may not base change well in the relative setting.


\begin{thebibliography}{100}
\bibitem{baily} W. L. Baily, Jr., \emph{Satake's compactification of $V_n$,} Amer. J. Math. \tbf{80} (1958), 348-364.
\bibitem{balaji} V. Balaji, \emph{Torsors on semistable curves and degenerations}, Proc. Indian Acad. Sci. Math. Sci. \tbf{132} (2022), no. 1, Paper No. 27, 63 pp.
\bibitem{beauville-laszlo} A. Beauville, Y. Laszlo, \emph{Conformal blocks and generalized theta functions}, Comm. Math. Phys. \tbf{164} (1994), no. 2, 385-419.
\bibitem{belfakh} P. Belkale, N. Fakhruddin, \emph{Triviality properties of principal bundles on singular curves,} Algebr. Geom. \tbf{6} (2019), no. 2, 234-259.
\bibitem{belgib} P. Belkale, A. Gibney, \emph{On finite generation of the section ring of the determinant of cohomology line bundle,}Trans. Amer. Math. Soc. \tbf{371} (2019), no. 10, 7199-7242.
\bibitem{qhorn} P. Belkale, S. Kumar, \emph{The multiplicative eigenvalue problem and deformed quantum cohomology,} Adv. Math. \tbf{288} (2016), 1309–1359.
\bibitem{bhosle} U. Bhosle, \emph{Generalised parabolic bundles and applications to torsionfree sheaves on nodal curves,} Ark. Mat. \tbf{30} (1992), no. 2, 187-215.
\bibitem{dm} P. Deligne, D. Mumford, \emph{The irreducibility of the space of curves of given genus,} Inst. Hautes Études Sci. Publ. Math. (1969), no. 36, 75-109.
\bibitem{fakh12} N. Fakhruddin, \emph{Chern classes of conformal blocks}, Compact Moduli Spaces and Vector Bundles, Contemp. Math., vol. 564, Amer. Math. Soc., Providence, RI, 2012, pp. 145-176.
\bibitem{faltingschai} G. Faltings and C. Chai, \emph{Degeneration of Abelian Varieties,} Ergebnisse der Mathematik und ihrer Grenzgebiete (3), vol. 22, Springer-Verlag, Berlin, 1990.
\bibitem{faltingsverlinde} G. Faltings, \emph{A proof for the Verlinde formula,} J. Algebraic Geom. \tbf{3} (1994), no. 2, 347-374.
\bibitem{faltings} G. Faltings, \emph{Moduli-stacks for bundles on semistable curves}, Math. Ann. \tbf{304} (1996), 489-515.
\bibitem{fgaexplained} B. Fantechi, L. Göttsche, L. Illusie, S. L. Kleiman, N. Nitsure, A. Vistoli, \emph{Fundamental Algebraic Geometry. Grothendieck's FGA Explained,} Mathematical Surveys and Monographs, Vol. 123, American Mathematical Society, Providence, RI, 2005.
\bibitem{fultonharris} W. Fulton, J. Harris \emph{Representation theory. A first course,} Graduate Texts in Mathematics, 129. Readings in Mathematics. Springer-Verlag, New York, 1991.
\bibitem{gomezlangerschmittsols} T. L. Gómez, A. Langer, A. H. W. Schmitt, I. Sols, \emph{Moduli spaces for principal bundles in arbitrary characteristic,} Adv. Math. \tbf{219} (2008), no. 4, 1177-1245.
\bibitem{gomez-sols} T. G\'{o}mez, I. Sols, \emph{Stable tensors and moduli space of orthogonal sheaves,} arXiv: 0103150 (2001).
\bibitem{EGA} A. Grothendieck, \emph{Éléments de géométrie algébrique. IV. Étude locale des schémas et des morphismes de schémas. II,} Inst. Hautes Études Sci. Publ. Math, (1965), no. 24, 231 pp.
\bibitem{hart80} R. Hartshorne, \emph{Stable reflexive sheaves,} Math. Ann. \tbf{254} (1980), no. 2, 121-176.
\bibitem{huylehn} D. Huybrechts, M. Lehn, \emph{The geometry of moduli spaces of sheaves,} Second edition. Cambridge Mathematical Library. Cambridge University Press, Cambridge, 2010. xviii+325 pp.
\bibitem{knr} S. Kumar, M. S. Narasimhan, A. Ramanathan, \emph{Infinite Grassmannians and moduli spaces of G-bundles,} Math. Ann. \tbf{300} (1994), no. 1, 41-75.
\bibitem{laszlo-sorger} Y. Laszlo, C. Sorger, \emph{The line bundles on the moduli of parabolic G-bundles over curves and their sections,} Ann. Sci. École Norm. Sup. (4) \tbf{30} (1997), no. 4, 499-525.
\bibitem{laumon-rapoport} G. Laumon, M. Rapoport, \emph{The Langlands lemma and the Betti numbers of stacks of G-bundles on a curve,} Internat. J. Math. \tbf{7} (1996), no. 1, 29-45. 
\bibitem{manon} C. Manon, \emph{The algebra of conformal blocks,} J. Eur. Math. Soc. \tbf{20} (2018), no. 11, 2685-2715. 
\bibitem{martens-thaddeus} J. Martens, M. Thaddeus, \emph{Compactifications of reductive groups as moduli stacks of bundles,} Compos. Math. \tbf{152} (2016), no. 1, 62–98.
\bibitem{moonyoo} H. Moon, S. Yoo, \emph{Finite generation of the algebra of type A conformal blocks via birational geometry II: higher genus,} Proc. Lond. Math. Soc. (3) \tbf{120} (2020), no. 2, 242-264.
\bibitem{mcthesis} A. L. Mu\~{n}oz-Casta\~{n}eda, \emph{Principal G-bundles on nodal curves}, Ph.D. thesis, Freie Universit\"{a}t Berlin (2017).
\bibitem{mcuniformity} A. L. Mu\~{n}oz-Casta\~{n}eda, \emph{A compactification of the universal moduli space of principal G-bundles,} arXiv: 1806.06300 (2020).
\bibitem{schmc} A. L. Mu\~{n}oz-Casta\~{n}eda, A. H. W. Schmitt, \emph{Singular principal bundles on reducible nodal curves,} arXiv: 1911.01578 (2020).
\bibitem{nagarajseshadriI} D. S. Nagaraj, C. S. Seshadri, \emph{Degenerations of the moduli spaces of vector bundles on curves,} I. Proc. Indian Acad. Sci. Math. Sci. \tbf{107} (1997), no. 2, 101-137.
\bibitem{nagarajseshadriII} D. S. Nagaraj, C. S. Seshadri, \emph{Degenerations of the moduli spaces of vector bundles on curves. II. Generalized Gieseker moduli spaces,} Proc. Indian Acad. Sci. Math. Sci. \tbf{109} (1999), no. 2, 165-201.
\bibitem{theta1} M.S. Narasimhan, T. R. Ramadas, \emph{Factorisation of generalised theta functions. I,} Invent. Math. \tbf{114} (1993), no. 3, 565-623.
\bibitem{pandharipande} R. Pandharipande, \emph{A compactification over $\bbar{M}_g$ of the universal moduli space of slope-semistable vector
bundles,} J. Amer. Math. Soc. \tbf{9} (1996), no. 2, 425-471.
\bibitem{ramaphd} A. Ramanathan, \emph{Moduli of principal bundles over algebraic curves I-II,} Proceedings of the Indian Academy of Science, \tbf{106} (1996), 301-328, 421-449.
\bibitem{schspgb} A. H. W. Schmitt, \emph{Singular principal G-bundles on nodal curves,} J. Eur. Math. Soc. \tbf{7} (2005), no. 2, 215-251.
\bibitem{schbh} A. H. W. Schmitt, \emph{Moduli spaces for semistable honest singular principal bundles on a nodal curve which are compatible with degeneration. A remark on U. N. Bhosle's paper: ``Tensor fields and singular principal bundles,''} Int. Math. Res. Not. (2005), no. 23, 1427-1437.
\bibitem{seshadrivb} C. S. Seshadri, \emph{Fibr\'{e}s Vectoriels sur les Courbes Alg\'{e}briques}, Asterisque, 96, Soci\'{e}t\'{e} Mat\'{e}matique de France, Paris, 1982.
\bibitem{simpson} C. Simpson, \emph{Moduli of representations of the fundamental group of a smooth projective variety. I,} Inst. Hautes Études Sci. Publ. Math, (1994), no 79, 47-129.
\bibitem{psolis} P. Solis, \emph{A complete degeneration of the moduli of $G$-bundles on a curve,} arXiv:1311.6847.
\bibitem{sun} X. Sun, \emph{Degeneration of moduli spaces and generalized theta functions,} J. Algebraic Geom. \tbf{9} (2000), no. 3, 459–527.
\bibitem{sun1} X. Sun, \emph{Degenerations of SL(n)-bundles on a reducible curve,} Proceedings of the Symposium on Algebraic Geometry in East Asia (2001), 3-10.
\bibitem{tuy}  A. Tsuchiya, K. Ueno, Y. Yamada, \emph{Conformal field theory on universal family of stable curves with gauge symmetries}, Integrable systems in quantum field theory and statistical mechanics, Adv. Stud. Pure Math., vol. 19, Academic Press, Boston, MA, 1989.
\bibitem{wang} J. Wang, \emph{The moduli stack of G-bundles,} arXiv:1104.4828 (2011).
\end{thebibliography}
\end{document}